%% file: main.tex
\documentclass{article}
\input{head}

\title{Graphs With the Same Edge Count in Each Neighborhood
}
\author{Nathan S. Sheffield and Zoe Xi}
\date{}

\begin{document}

\maketitle
\begin{abstract}
    In a recent paper, Caro, Lauri, Mifsud, Yuster, and Zarb ask which
    parameters $r$ and $c$ admit the existence of an $r$-regular graph such that
    the neighborhood of each vertex induces exactly $c$ edges.  They
    show that every $r$ with $c$ satisfying $0\leq c\leq {r\choose
      2}-5r^{3/2}$ is achievable, but no $r$ with $c$ satisfying
    ${r\choose 2}-\lfloor\frac{r}{3}\rfloor\leq c\leq {r\choose 2}-1$
    is.  We strengthen the bound in their nonexistence result from
    ${r\choose 2}-\lfloor\frac{r}{3}\rfloor$ to ${r\choose
      2}-\lfloor\frac{r-2}{2}\rfloor$.  Additionally, when the graph
    is the Cayley graph of an abelian group, we obtain a much more
    fine-grained characterization of the achievable values of $c$
    between $\binom{r}{2} - 5r^{3/2}$ and $\binom{r}{2} -
    \floor{\frac{r-2}{2}}$, which we conjecture to be the correct
    answer for general graphs as well.
    That result relies on a lemma about approximate subgroups in the ``99\% regime,'' quantifying the extent to which nearly-additively-closed subsets of an abelian group must be close to actual subgroups.
    Finally, we consider a generalization to graphs with
    multiple types of edges and partially resolve several open
    questions of Caro et al. about \emph{flip} colorings of graphs.

\end{abstract}

\input{intro-flipmoderate}

\input{prelims}
\input{rc-constant}
\input{lp}
\input{hardness}


\section{Acknowledgements}
The results in this paper were produced as a part of the Duluth REU program; we thank Joe Gallian and Colin Defant for running the program. 
We thank the organizers, advisors, visitors and participants of the program for advice and providing a stimulating research environment.
We are especially grateful to Maya Sankar, Mitchell Lee, and Alek Westover for helpful discussions, and to Alec Sun and Joe Gallian for comments on the manuscript.
Funding was provided by NSF grant no. DMS-2409861, Jane Street Capital, and personal contributions by Ray Sidney and Eric Wepsic.

\printbibliography
\end{document}

%% file: head.tex
\usepackage[left=1in, right=1in, top=1in, bottom=1in]{geometry}

\usepackage{amsthm}
\usepackage{amssymb}
\usepackage{amsmath}
\usepackage{mathtools}
\usepackage{csquotes}
\usepackage{enumitem}
\usepackage{svg}
\usepackage{hyperref}
\usepackage{xspace}
\usepackage{xcolor}
\usepackage{fancyhdr}
\usepackage{complexity}

\setlist[itemize]{leftmargin=*}
\setlist{nosep}
\setlist{nolistsep}

\hypersetup{
    colorlinks,
    linkcolor={red!50!black},
    citecolor={blue!100!black},
    urlcolor={black!80!black}
}

\newcommand{\defn}[1]{{\textit{\textbf{\boldmath #1}}}\xspace}
\renewcommand{\paragraph}[1]{\vspace{0.04in}\noindent{\bf \boldmath #1:}}

\usepackage{bbm}

\DeclareMathOperator{\Triv}{Triv}

\newcommand{\interior}[1]{ {\kern0pt#1}^{\mathrm{o}} }

\newcommand{\eps}{\varepsilon}
\renewcommand{\Re}{\mathrm{Re}}

\newcommand{\Z}{\mathbb{Z}}

\newcommand{\N}{\mathbb{N}}
\renewcommand{\C}{\mathbb{C}}

\renewcommand{\E}{\mathop{\mathbb{E}}}

\usepackage{algorithm}
\usepackage[noend]{algpseudocode} 
\makeatletter
\newcounter{HALG@line}
\renewcommand{\theHALG@line}{\thealgorithm.\arabic{ALG@line}}
\makeatother

\newcommand{\floor}[1]{\left\lfloor #1 \right\rfloor}
\newcommand{\ceil}[1]{\left\lceil #1 \right\rceil}

\usepackage[capitalise,nameinlink,noabbrev]{cleveref}
\crefname{equation}{}{} 
\crefname{enumi}{Step}{} 

\theoremstyle{definition}

\newtheorem{defin}{Definition}
\newtheorem{rmk}{Remark}

\newtheorem{prop}{Proposition}
\newenvironment{proposition}{\begin{prop}}{\end{prop}}
\newtheorem{lem}{Lemma}
\newenvironment{lemma}{\begin{lem}}{\end{lem}}
\newtheorem{clm}{Claim}
\newenvironment{claim}{\begin{clm}}{\end{clm}}
\newtheorem{question}{Question}
\newtheorem{cor}{Corollary}

\newtheorem{thm}{Theorem}
\newenvironment{theorem}{\begin{thm}}{\end{thm}}
\newtheorem{conj}{Conjecture}
\newenvironment{conjecture}{\begin{conj}}{\end{conj}}
\newtheorem{prob}{Problem}

\makeatletter
\newtheorem*{rep@theorem}{\rep@title}
\newcommand{\newreptheorem}[2]{%
    \newenvironment{rep#1}[1]{%
        \def\rep@title{\cref{##1}}%
        \begin{rep@theorem}}%
            {\end{rep@theorem}}}
\makeatother
\newreptheorem{theorem}{Theorem}

\usepackage{mdframed}
\usepackage{framed}
\usepackage{thm-restate}

\usepackage[
    backend=biber,
    style=alphabetic,
    maxbibnames=99,
]{biblatex}
\DeclareFieldFormat[article,inbook,incollection,inproceedings,patent,thesis,unpublished]{title}{#1}
\DeclareFieldFormat{pages}{#1}
\renewbibmacro{in:}{}

\setlist[enumerate]{itemsep=\smallskipamount,parsep=0pt,label={\rm \roman*)}}
\setlist[itemize]{itemsep=\smallskipamount,parsep=0pt}

%% file: intro-flipmoderate.tex
\section{Introduction}
The question of what it means for a graph to ``look the same'' everywhere locally has been extensively studied in structural graph theory.
One weak version of this notion is to require a graph to be $r$-regular --- in other words, to specify that the neighborhood of each vertex is the same size.
A stronger notion is that of a \defn{constant-link graph}. 
In a constant-link graph, not only must each vertex's neighborhood contain the same number of vertices, they must all induce the same subgraph (the ``link'' of the graph).
A substantial literature aims to understand the structure of such graphs.
Arguably the most fundamental question, known as the Trahtenbrot--Zykov problem, is to characterize which graphs can be the link of some constant-link graph.
Answers are known for a few classes of graphs, such as trees, cycles, and graphs on six or fewer vertices~\cite{brown1975graphs,blass1980trees,hall1985graphs}.
However, a result of Buklito shows that the problem of determining whether a general graph $H$ can be the link of a constant-link graph is algorithmically undecidable~\cite{bulitko1973graphs}.\\

In this work, we consider the less-restrictive class of \defn{$(r,c)$-triangle-regular} graphs.
Here, we require that the neighborhood of each vertex has size $r$ and induces exactly $c$ edges, but we do not require the neighborhoods to be isomorphic.
In other words, each vertex has degree $r$ and is involved in exactly $c$ distinct triangles.
The analogous question to the Trahtenbrot--Zykov problem for this notion would be: ``For which $r$ and $c$ does an $(r,c)$-triangle-regular graph exist?''.
Since we are only concerned with total edge counts in each neighborhood and not the exact subgraph, one might hope that a characterization for this problem is more tractable.
And indeed, Caro, Lauri, Mifsud, Yuster, and Zarb show the following theorem.

\begin{theorem}[\cite{flip1}]\label{thm:caro-et-al-rc}

    Let $r, c \in\mathbb{N}$. We have the following.
    
    \begin{enumerate}
    
    \item If $0\leq c\leq {r\choose
      2}-5r^{3/2}$, then there exists an $(r, c)$-triangle-regular graph.
    
    \item If ${r\choose
      2}-\lfloor\frac{r}{3}\rfloor\leq c\leq {r\choose 2}-1$, then there does
      not exist an $(r, c)$-triangle-regular graph.
      
    \end{enumerate}

\end{theorem}

Since $r$ vertices induce $\le {r\choose 2}$ edges, it is clear that
no $(r,c)$-triangle-regular graph can exist with $c >
\binom{r}{2}$. Note that $c = \binom{r}{2}$ is achieved by the
complete graph $K_{r+1}$.  Thus, \cref{thm:caro-et-al-rc} gives an
exact characterization of the possible values of $c$ outside an
interval of length $O(r^{3/2})$.  The question remains to understand
which values are possible within that interval. In this paper, we
present a tight lower bound on the values of $c$ for which there does
not exist an $(r, c)$-triangle-regular graph and fully characterize
the values of $c$ for which there exists an $(r, c)$-triangle-regular
abelian Cayley graph. \\

\subsection{Main results}

Our first result, in \cref{sec:zoe-upperbound}, is a strengthening of the bounds of \cref{thm:caro-et-al-rc}.

\begin{restatable}{theorem}{rcnonexistence}\label{thm:rc-nonexistence}
    Given a positive integer $r$, for every integer $c$ such that
    ${r\choose 2}-\lfloor\frac{r-2}{2}\rfloor\leq c\leq {r\choose 2}-1$,
    there does not exist an $(r, c)$-triangle-regular graph.   
\end{restatable}

This bound is tight for even $r$: if $r \equiv 0 \text{ mod 2}$ then there \emph{does} exist an $\left(r, \binom{r}{2} - \frac{r}{2}\right)$-triangle-regular graph.\\

The question that now remains is: for which $c$ with $\binom{r}{2} - 5r^{3/2} < c < \binom{r}{2} - \floor{\frac{r-2}{2}}$ is there an $(r,c)$-triangle-regular graph?
Is there some threshold below which all values are achievable and above which all values are inachievable, or is the interval more spotty?
Is there a value on the order of $\binom{r}{2} - \Omega(r^{3/2})$ that is not achievable, or can we substantially strengthen the existence part of \cref{thm:caro-et-al-rc}?
In \cref{sec:cayley}, we answer these questions for the special case of abelian Cayley graphs.

\begin{restatable}{theorem}{cayleyconstruction}\label{thm:cayley-construction}
    For every $r\in \N$, and every $x, y \in \N$ with $x \leq r$ and $\frac{x^2}{8} + 3x \leq y \leq \frac{x^2}{4}-4x^{3/2}$,
    there exists an $(r,c)$-triangle-regular abelian Cayley graph with $c = \binom{r}{2} - \frac{rx}{2} + y$.
\end{restatable}

\begin{restatable}{theorem}{cayleynonexistence}\label{thm:rc-cayley-nonexistence}
    There exists some uniform constant $C$ such that, if there exists an $(r,c)$-triangle-regular abelian Cayley graph, then $c = \binom{r}{2} - \frac{rx}{2} + y$ for some $x$ and for some $y \in \left\{ -Cx^2,\dots, Cx^2\right\}$.
\end{restatable}

In particular, this means that the interval is spotty, and that the smallest non-achievable value of $c$ is $\binom{r}{2} - \Omega(r^{3/2})$. The key technical ingredient in the proof of \cref{thm:rc-cayley-nonexistence} is the following lemma, demonstrating that subsets of an abelian group which are ``close'' to being closed under addition are also ``close'' to some subgroup:

\begin{restatable}{lemma}{subgrouptestable}\label{lem:subgroup-testable}
    There exists some uniform constant $K$ such that, for any finite abelian group $G$, any subset $S \subseteq G$, and any $\varepsilon > 0$, if $\Pr_{a, b \in S}[a + b \in S] \ge 1 - \varepsilon$, then there exists some subgroup $H \leq G$ such that $|H| \leq (1 + K\varepsilon)|S|$, and $|H \cap S| \geq (1 - K\varepsilon)|S|$. 
\end{restatable}

The restriction to abelian Cayley graphs is important, as our proof of \cref{lem:subgroup-testable} relies on Fourier analysis --- however, it seems conceivable that \cref{lem:subgroup-testable}, and hence \cref{thm:rc-cayley-nonexistence}, hold for arbitrary groups. In fact, one might even suspect the following, implying that these answers extend to arbitrary graphs:

\begin{conjecture}\label{conj:abelian-Cayley}
If there exists an $(r, c)$-triangle-regular graph, then there exists
an $(r, c)$-triangle-regular abelian Cayley graph.
\end{conjecture}  

The intuition behind such a conjecture is that it would be surprising for triangle-regularity to be possible only by exploiting some lack of symmetry, so one might expect that any triangle-regular graph can be made highly symmetric.\\

In addition to this study of $(r,c)$-triangle-regular graphs, we next consider a generalization to graphs with multiple types of edges.

\begin{restatable}{defin}{multicolor}\label{def:multicolor}
Fix some $k \in \N$ and $\vec{r}, \vec{c} \in \N^k$. A $t$-edge-colored graph $G$ is \defn{$(\vec{r}, \vec{c})$-triangle-regular} if every vertex is incident to exactly $\vec{r}[i]$ edges of color $i$, and the neighborhood of each vertex induces exactly $\vec{c}[i]$ edges of color $i$, for all $i \in \{1,\dots, t\}$.
\end{restatable}

In this setting, it is much harder to characterize which $\vec{r},\vec{c}$ admit a $(\vec{r}, \vec{c})$-triangle-regular graph.
However, in \cref{sec:lp}, we do present a simple linear program whose feasibility is a necessary condition for the existence of such a graph.
We also show that, if there is a feasible solution to that linear program, then there exists a $\left(\vec{r'}, \vec{c'}\right)$-triangle-regular graph, where $\vec{r'}$ and $\vec{c'}$ differ by at most a constant factor in each coordinate from $\vec{r}$ and $\vec{c}$, respectively.
This can be thought of as providing an efficient algorithm for distinguishing between $\vec{r}, \vec{c}$ that admit a $(\vec{r}, \vec{c})$-triangle-regular graph, and $\vec{r}, \vec{c}$ that are far from any pair of vectors admitting such a graph. \\

Finally, in \cref{sec:hard}, on the algorithmic \emph{infeasibility} side, we show that given an uncolored graph $G$, deciding whether there exists an $(\vec{r}, \vec{c})$-triangle-regular edge-coloring is $\mathsf{NP}$-hard in general.

\begin{restatable}{theorem}{rchard}\label{thm:rchard}
    Given a graph $G$, and vectors $\vec{r}, \vec{c} \in \N^2$, it is $\NP$-complete to determine whether $G$ can be $2$-colored to yield an $(\vec{r}, \vec{c})$-triangle regular graph.
\end{restatable}

A related notion to $(\vec{r}, \vec{c})$-triangle regular graphs, introduced by Caro et al. \cite{flip1}, is that of a \defn{flip coloring} of a graph: an edge-coloring with a sequence of colors such that the degrees at each vertex are strictly increasing along the sequence of colors, but the prevalence of edges in the closed neighborhood of each vertex are strictly decreasing along the sequence of colors. We are able to use the machinery we develop to address several open questions posed in their work: in \cref{sec:lp}, we improve several of their bounds on ``flip sequences,'' and in \cref{sec:hard} we show that a variant of our $(\vec{r}, \vec{c})$-triangle-regular hardness proof gives $\NP$-hardness for finding flip colorings.


\subsection{Related work}

The term $(r,c)$-triangle-regular graph was first introduced by Nair and Vijayakumar under the name ``strongly vertex triangle regular graph''~\cite{nair1994triangles, nair1996strongly}.
They have also been referred to as ``$(r,c)$-constant graphs'' in a recent sequence of papers that aims to characterize which $r$ and $c$ are achievable, with particular interest in the related problem of flip coloring~\cite{flip1,rc-constant,flip2,mifsud2024local}.
The results and techniques of our paper build on the papers \cite{flip1,rc-constant,flip2,mifsud2024local}; however, we use the term ``$(r,c)$-triangle-regular'' as a compromise between these notations.\\

We show in \cref{sec:cayley} that the existence of $(r,c)$-triangle-regular abelian Cayley graphs can be viewed as an additive combinatorics question: we want to know for which values of $c$ there exist an abelian group $G$ and a symmetric subset $S$, with $|S| = r$ and $S$ containing exactly $2c$ additive triples (i.e., ordered triples $x, y, z \in S$ such that $x+y = z$).
There is a line of work considering for particular abelian groups the possible counts of additive triples in subsets of a given size~\cite{samotij2016number,huczynska2024additive}.
Our problem differs in that we are not concerned with any particular abelian group, but instead aim to characterize which $c$ are possible for a size-$r$ subset of \emph{some} abelian group.\\

Our \cref{lem:subgroup-testable}, about the degree to which nearly-additively-closed subsets of an abelian group correspond to actual subgroups, belongs to an extensive body of existing work on ``approximate groups.'' There are a number of different ways to quantify what it means for a subset of a group to be ``approximately'' closed under the group operation --- for instance, one can specify that the set has small doubling or tripling, or large energy. Most existing work considers the ``1\% regime,'' i.e., where these values are off by a constant factor from what they would be in a true subgroup, in which case most natural notions turn out to be roughly equivalent. \cref{lem:subgroup-testable} belongs instead to the ``99\% regime,'' in that we require this constant factor to be very close to $1$. In this setting, known results imply qualitative analogues of \cref{lem:subgroup-testable} --- it was already known that one can find a subgroup with $|H| \leq (1 + K \sqrt{\eps})|S|$ and $|H \cap S| \geq (1-K\sqrt{\eps})|S|$~\cite{fournier1977sharpness, sanders2012approximate} --- but as far as we are aware a linear dependence on $\eps$ is novel. Some further discussion of related results is given in \cref{sec:cayley}; for surveys on approximate subgroups, the reader is referred to introductions by Breuillard or Tointon~\cite{breuillard2014brief, tointon2019introduction}.\\

A complementary notion to that of a triangle-regular graph is a triangle-\emph{distinct} graph --- that is, a graph where \emph{no} two vertices are involved in the same number of triangles. Erd\H{o}s and Trotter initiated the study of such graphs --- bounds and constructions for triangle-distinct graphs have been investigated recently as well~\cite{berikkyzy2024triangle, stevanovic2024regular}.

%% file: prelims.tex
\section{Preliminaries}

Our results in \cref{sec:cayley} will make use of Fourier analysis over abelian groups; readers unfamiliar with terminology are referred to~\cite{o2014analysis}.\\

We generally use arrows to denote vector quantities in order to distinguish from scalars.
In particular, ``$(\vec{r}, \vec{c})$-triangle-regular'' denotes the multicolored generalization, where $\vec{r}$ and $\vec{c}$ are both length-$t$ vectors, while the standard notion is denoted ``$(r,c)$-triangle-regular''.
For a vector $\vec{v}$, we write $\vec{v}[i]$ to denote the $i$th entry.\\

We denote the neighborhood of a vertex by $N(v)$. We also occasionally use $N[v]$ to denote the ``closed neighborhood'' of a vertex; that is, $N[v] = N(v) \cup \{v\}$.\\

The \defn{Cartesian product} between graphs $G$ and $H$, denoted $G \square H$, has vertex set $V(G) \times V(H)$, and edge set $\left\{ \big((u,v), (u',v) \big) \ | \ (u,u') \in E(G) \right\} \cup \left\{ \big((u,v), (u,v') \big) \ | \ (v,v') \in E(H) \right\}$.
When $G$ and $H$ are edge-colored, let the color of an edge in their Cartesian product be the same as the color of the edge it corresponds to in $G$ or $H$.
Observe that, if $G$ is an $(\vec{r}, \vec{c})$-triangle-regular graph, and $H$ is an $\left(\vec{r'}, \vec{c'}\right)$-triangle-regular graph, then $G \square H$ is an $\left(\vec{r} + \vec{r'}, \vec{c} + \vec{c'}\right)$-triangle-regular graph; this fact is used extensively in both our constructions and those of Caro et al.~\cite{flip1}. \\

%% file: rc-constant.tex
\section{Values of $c$ for which there does not exist an $(r,c)$-triangle-regular graph}\label{sec:zoe-upperbound}

First, we extend the interval of values for which
there does not exist an $(r, c)$-triangle-regular graph given in \cref{thm:caro-et-al-rc}.

\rcnonexistence*

\begin{proof}

Let $k\leq (r-2)/2$ be a positive integer. Suppose for the sake of
contradiction that there exists an $(r, c)$-triangle-regular graph
with $c = {r\choose 2}-k$.\\

Let $G = (V, E)$ be such an $(r, c)$-triangle-regular graph, and let $v\in V$
be a vertex. Let $R(v)$ be the set of vertices in $N[v]$ whose
neighborhood is contained in $N[v]$; let $S(v) = N[v]\backslash R(v)$
(that is, $S(v)$ is the set of vertices in $N[v]$ that have a neighbor
outside of $N[v]$); and let $T(v)\subseteq V\backslash N[v]$ be the
set of vertices that are not in $N[v]$ but have a neighbor in
$N[v]$. We consider three cases. \\

\paragraph{Case 1} There is a vertex $w\in S(v)$ such that $w$ has at least
$(k+1)/3$ neighbors in $T(v)$. Consider $N(w)$; we have that
$R(v)\subseteq N(w)$ as $w\in N[v]$. Since there are $r+1$ vertices
in $N[v]$ and the $k$ edges missing from the would-be $(r+1)$-clique
touch at most $2k$ vertices, $|R(v)|\geq r+1-2k$. As there are no
edges between vertices in $T(v)$ and vertices in $R(v)$, this
implies that there are at most ${r\choose 2}-(r+1-2k)(k+1)/3\leq
{r\choose 2}-k-1$ edges between vertices in $N(w)$, which
contradicts that $G$ is $(r, c)$-triangle-regular. \\

\paragraph{Case 2}
There is a vertex $x\in T(v)$ such that $x$ has at least $2(k+1)/3$
neighbors in $S(v)$. We have that every vertex $w\in
N(x)\cap S(v)$ is neighbors with every vertex in $R(v)$ and that
$R(v)\cap N(x) = \varnothing$ as $x\notin N[v]$; since $R(v)\geq
r+1-2k$, this implies that $w$ is neighbors with at most $2k-1$
vertices in $N(x)$. This in turn implies that there are at most
\[{r\choose 2}-\frac{2(k+1)}{3}\cdot\frac{r+1-2k}{2}\leq {r\choose 2}-k-1
\]
edges between vertices in $N(x)$, which contradicts that $G$ is $(r,
c)$-triangle-regular. \\

\paragraph{Case 3}
For the third case, assume that we are neither in case 1 nor 2. Let
$m$ be the size of $S(v)$, and note that $R(v) = r+1-m$. Let $x\in
T(v)$, and let $w$ be a neighbor of $x$ in $S(v)$. Since we are not in
case 2, we know that $x$ is not neighbors with at least $m-2(k+1)/3+1$
vertices in $S(v)$. Let us call this set of vertices $A$. Out of the
vertices in $A$, at least $m-2(k+1)/3+1-(k+1)/3+1 = m-k+1$ are
neighbors of $w$, since we are not in case 1 and thus $w$ has at most
$(k+1)/3-1$ neighbors not in $N[v]\supseteq A$. Consider $N(w)$. As
there are no edges between $x$ and vertices in $R(v)\cup A$ and
$R(v)\cap A = \varnothing$, this implies that there are at most
${r\choose 2}-(r+1-m+m-k+1)\leq {r\choose 2}-k-1$ edges between
vertices in $N(w)$, which contradicts that $G$ is $(r,
c)$-triangle-regular. \\
  
We conclude that there does not exist an $(r, c)$-triangle-regular
graph for ${r\choose 2}-\lfloor (r-2)/2\rfloor\leq c\leq {r\choose
  2}-1$.

\end{proof}

\section{Possible values of $c$ in abelian Cayley graphs}\label{sec:cayley}

We now know that there exists an $(r,c)$-triangle-regular graph for all $c$ in $\{1, \dots, \binom{r}{2} - 5r^{3/2}\}$, and there does not exist an $(r,c)$-triangle-regular graph for any $c$ in $\{\binom{r}{2} - \lfloor\frac{r-2}{2}\rfloor, \dots, \binom{r}{2} - 1\}$. However, it remains unclear which $c$ are achievable in $\{\binom{r}{2} - 5 r^{3/2}, \dots, \binom{r}{2} - \lfloor\frac{r-2}{2}\rfloor\}$. 
One might wonder whether there is some threshold such that all smaller values can be made and all larger values cannot, or whether the interval has gaps.\\

In this section, we answer this question when the $(r,c)$-triangle-regular graph is the Cayley graph of a finite abelian group. Note that the previous constructions by Caro et al. for $c \in \{1, \dots, \binom{r}{2} - 5r^{3/2}\}$ are Cartesian products of cliques, and hence abelian Cayley graphs \cite{flip1}. It seems natural for $(r,c)$-triangle-regular graphs, when they exist, to be highly symmetric, so we conjecture that, whenever there exists an $(r,c)$-triangle-regular graph, there also exists an $(r,c)$-triangle-regular abelian Cayley graph, in which case the results of this section would be close to a complete picture,

\subsection{Construction of $(r,c)$-triangle-regular abelian Cayley graphs}

We show an explicit construction of $(r,c)$-triangle-regular graphs, showing that there are intervals of achievable values of $c$ close to every multiple of $\frac{r}{2}$. 

\cayleyconstruction*

\begin{proof}
    The basic idea is to take the Cartesian product of a clique (or something very close to a clique) $G_1$ with a smaller, less dense but more fine-tuned Cayley graph $G_2$.
    To do so, we need to split into a few cases due to parity constraints.

    \paragraph{Case 1} $x$ is even. 
    In this case, let $G_1$ be a clique on $r - x/2 + 1$ vertices, and let $G_2$ be some degree-$(x/2)$ abelian Cayley graph.
    Letting $k$ be the number of edges in each open neighborhood of $G_2$, the Cartesian product of $G_1$ and $G_2$ is an $r$-regular abelian Cayley graph with
    \[\binom{r-x/2}{2} + k 
    = \binom{r}{2} - \frac{rx}{2} + \frac{x^2+2x}{8} + k\]
    edges in each open neighborhood. By \cref{thm:caro-et-al-rc}, there exists a choice of $G_2$ achieving any value of $k$ in the interval $\left\{0, \dots, \floor{\binom{x/2}{2} - 5\left(x/2\right)^{3/2}}\right\}$.
    Thus, every $y$ in the desired range is achievable.\\

    \paragraph{Case 2} $2r - x \equiv 3 \text{ mod 4}$. To construct $G_1$, take the underlying group to be $A = \Z / \big((2r - x + 5)/2\big)\mathbb{Z}$, with generating set $S = A \setminus \{0, (2r - x + 5)/4\}$.
    Note that $A$ is symmetric because $(2r - x + 5)/4 \equiv -(2r - x + 5)/4 \text{ mod } (2r - x + 5)/2.$ 
    The resulting Cayley graph can be described as a clique on $(2r - x + 5)/2$ vertices, with the edges of a perfect matching removed. This is an $\left(r - \frac{x-1}{2} \right)$-regular graph, with 
    \[\binom{r - \frac{x-1}{2}}{2} - \frac{2r - x + 1}{4} = 
    \binom{r}{2} - \frac{rx}{2} + \frac{x^2-2x+1}{8}\]     
    edges in each open neighborhood. 
    Again, by \cref{thm:caro-et-al-rc}, for any $k \in \left\{0, \dots, \floor{\binom{(x-1)/2}{2} - 5\left(\frac{x-1}{2}\right)^{3/2}}\right\}$, there exists a $\left(\frac{x-1}{2}, k\right)$-triangle-regular abelian Cayley graph $G_2$.
    Taking the Cartesian product of such $G_2$ with $G_1$ yields the entire desired range.\\

    \paragraph{Case 3} $2r - x \equiv 1 \text{ mod 4}$. In this case, let $A = \Z/\big((2r-x+11)/2\big)$, and $S = A \setminus \{0, (2r-x+7)/4, (2r-x+11)/4, (2r-x+15)/4\}$.
    Observe that $S$ is symmetric. Note that the statement of the theorem is trivially satisfied when $2r - x = 1$, so we can assume $|A| > 6$, which for case $\{(2r-x+7)/4, (2r-x+11)/4, (2r-x+15)\}$ is sum-free. 
    So, the Cayley graph generated by $S$ looks like a clique on $(2r-x+11)/2$ vertices, with the edges of a triangle-free 3-factor removed. This is an $\left(r - \frac{x-3}{2} \right)$-regular graph, with
    \[\binom{r - \frac{x-3}{2}}{2} - \frac{3\left(2r-x-1\right)}{4} = \binom{r}{2} - \frac{rx}{2} + \frac{x^2-2x+21}{8}\] 
    edges in each open neighborhood. Once again, this lets us generate an $(r,c)$-triangle-regular graph for any $c$ as in the theorem statement by taking the Cartesian product with a $\left(\frac{x-3}{2}, k\right)$-triangle-regular abelian Cayley graph for some $k \in \left\{0, \dots, \floor{\binom{(x-3)/2}{2} - 5\left(\frac{x-3}{2}\right)^{3/2}}\right\}$.
\end{proof}

\subsection{Nonexistence of $(r,c)$-triangle-regular abelian Cayley graphs}

As we will demonstrate, \cref{thm:cayley-construction} is ``approximately'' tight (where we define approximately later). That is, we show that all values of $c$ achievable by $(r,c)$-triangle-regular abelian Cayley graphs are ``close'' to multiples of $\frac{r}{2}$, where the closeness requirement becomes more lenient the further away one gets from $\binom{r}{2}$. Specifically, we show the following result.

\cayleynonexistence*

Note that \cref{thm:rc-cayley-nonexistence} does not rule out any values of $c$ with $c < \binom{r}{2} - \frac{r^{3/2}}{C^{1/2}}$, since for $b > \left(\frac{r}{C}\right)^{1/2}$ the possible range of $y$ is greater than $\frac{r}{2}$, and so every value of $c$ can be made. However, as $c$ becomes progressively closer to ${r\choose 2}$, \cref{thm:rc-cayley-nonexistence} excludes all values outside of progressively smaller intervals around each multiple of $\frac{r}{2}$. See \cref{fig:spec-numberline} for an illustration.\\

\begin{figure}[h]
  \centering
  \includegraphics[width=0.8\linewidth]{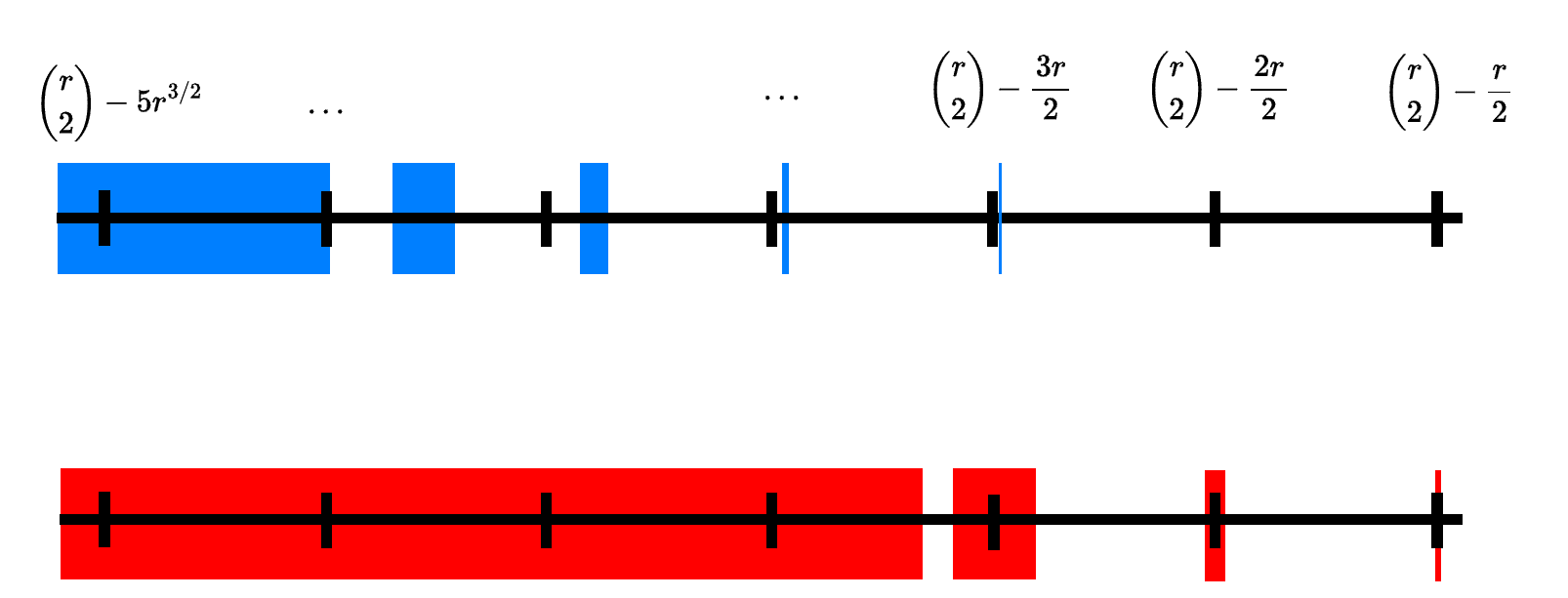}
  \caption{A sketch of the values of $c$ in between $\binom{r}{2} - 5r^{3/2}$ and $\binom{r}{2}-\frac{r}{2}$ achievable by $(r,c)$-triangle-regular abelian Cayley graphs. Top: by \cref{thm:cayley-construction}, there exists an $(r,c)$-triangle-regular abelian Cayley graph achieving every $c$ in these blue intervals. Bottom: by \cref{thm:rc-cayley-nonexistence}, for any $(r,c)$-triangle-regular abelian Cayley graph, $c$ must lie in these red intervals. In both cases, observe that the allowable values lie in intervals near multiples of $\frac{r}{2}$, with the length of those intervals quadratically increasing every step away from $\binom{r}{2}$, so that very few values are possible close to $\binom{r}{2}$, but all values are possible below $\binom{r}{2} - \Theta\left(r^{3/2}\right)$, since the intervals below there begin to overlap.}
  \label{fig:spec-numberline}
\end{figure}
The key technical tool in the proof of \cref{thm:rc-cayley-nonexistence} is a lemma, which may be of independent interest, that symmetric subsets of an abelian group that are \emph{close} to being closed under addition are correspondingly \emph{close} to some subgroup.

\subgrouptestable*

This lemma gives a strong constraint on the structure of the generating set of an $(r,c)$-triangle-regular Cayley graph for large $c$, allowing us to directly deduce \cref{thm:rc-cayley-nonexistence}.

\begin{proof}[Proof of \cref{thm:rc-cayley-nonexistence} given \cref{lem:subgroup-testable}]

Let $S\subseteq G$ be the generating set of an $(r,c)$-triangle-regular abelian Cayley graph.
The condition that every open neighborhood induces exactly $c$ edges is equivalent to the condition that there are exactly $2c$ ordered pairs $(a,b) \in S \times S$ such that $a+b \in S$. 
In other words, letting $\varepsilon = 1 - \frac{2c}{r^2}$, we have $\Pr_{a,b \in S}[a+b \in S] = 1 - \varepsilon$.\\

Now, by \cref{lem:subgroup-testable}, we know that, for some integers $\alpha, \beta \leq K\varepsilon r$, there exists a subgroup $H \subseteq G$ with  $|H \cap S| = r - \alpha$, and $|H| = r + \beta$. 
We now bound $c$ in terms of $\alpha$ and $\beta$. Note that every additive triple involving at least one element outside of $H$ must involve at least two, since $H$ is a subgroup.
So, we can upper bound the number of ordered pairs $(a,b) \in \left((S \setminus H)\times S\right) \cup \left(S \times (S \setminus H)\right)$ with $a + b \in S$ by $|S \setminus H|^2 = \alpha^2$.
The remaining ordered pairs $(a,b)\in S\times S$ with $a+b$ in $S$ have $a$, $b$, and $a+b$ all in $(S \cap H)$. \\

The total number of pairs $(a,b) \in H \times H$ with $a+b \not\in S$ is exactly $|H \setminus S| \cdot |H| = (\alpha+\beta)(r + \beta)$, since there is one choice for $b$ given any values of $a \in H$ and $a+b \in (H \setminus S)$.
The fact that $a$ and $a+b$ determine $b$ also lets us upper bound the number of pairs $(a,b) \in ((H \setminus S) \times H)$ with $a+b \not\in S$ by $|H \setminus S|^2 = (\alpha+\beta)^2$.
The same bound holds for the pairs $(a,b) \in (H \times (H\setminus S))$ with $a+b \not\in S$.
So, the total number of pairs $(a,b) \in (H\cap S) \times (H\cap S)$ with $a+b \in S$ is at least 
\[|H \cap S|^2 - \left( (\alpha + \beta)(r+ \beta)\right) = r^2 - (3\alpha + \beta)r + \alpha^2 - \alpha\beta - \beta^2  ,\]
and at most
\[|H\cap S|^2 - \left( (\alpha + \beta)(r+ \beta) - 2(\alpha+\beta)^2\right) = r^2 - (3\alpha + \beta)r + 3\alpha^2 + 3\alpha\beta + \beta^2.\]

Combining these bounds with our bound on the number of additive triples outside of $H$, we get that
\[ r^2 - (3\alpha + \beta)r + \alpha^2 - \alpha\beta - \beta^2 \leq 2c \leq r^2 - (3\alpha + \beta)r + 4\alpha^2 + 3\alpha\beta + \beta^2.\]

Thus, $c$ is within an additive $\alpha^2 + 3\alpha\beta + \beta^2 \leq 5 K^2\varepsilon^2r^2$ from a multiple of $\frac{r}{2}$. 
Letting $x$ be the integer minimizing distance from $\binom{r}{2} - \frac{rx}{2}$ to $c$, note that we have $\varepsilon \leq 1 - \frac{2\left(\binom{r}{2} - \frac{r(x-1)}{2}\right)}{r^2} \leq \frac{x+3}{r}$. 
So, $c$ is within an additive $5K^2(x+3)^2$ from $\binom{r}{2}-\frac{rx}{2}$.
Since our desired statement follows directly from \cref{thm:rc-nonexistence} when $x=0$, we can assume $x \geq 1$, in which case we have $5K^2(x+3)^2 \leq (80K^2)x^2$. 
So, we have shown that for $C = 80K^2$, there exists some integers $x$ and $y$, with $|y| \leq Cx^2$, such that $c = \binom{r}{2} - \frac{rx}{2} + y$.

\end{proof}

The remainder of this section is devoted to the proof of \cref{lem:subgroup-testable}. 
The basic idea can be thought of as two layers of a ``correction'' argument of a form common in property testing problems, where one shows that being reasonably close to some property actually implies being extremely close. 
On the outer layer, we show that we can repeatedly pass to subgroups containing most of $S$, where this ``correction'' idea prevents parameter loss in what ``most'' means.
On the inner layer, to attain that passage to a subgroup, we proceed using Fourier analysis, finding some character highly correlated with $S$ and repeatedly correcting to force most of $S$ to have image within successively smaller intervals around $1$ under that character.
We begin with this inner layer, finding a character in \cref{lem:buff-character} and then correcting most of its image to a single point in \cref{lem:pass-to-subgroup}.

\begin{lemma}\label{lem:buff-character}
  Let $G$ be an abelian group, and $S \subseteq G$ be a symmetric subset such that $|S| < (1-20000\varepsilon)|G|$, and $\Pr_{a,b \in S}[a+b \in S] \geq 1 - 10\varepsilon$.
  Then, there exists some nontrivial character $\chi: G \to \C^\times$ such that, for at least $.999|S|$ distinct values of $x \in S$, we have $\Re(\chi(x)) > 0$.
\end{lemma}
\begin{proof}
  Let $1_S: G \to \{0,1\}$ be the indicator function of $S$.
  We have $\mathbb{E}_{a, b\in S}[1_S(a+b)] \geq 1 - \varepsilon$, so equivalently $\sum_{a,b \in G}[1_S(a)1_S(b)1_S(a+b)] \geq (1 - 10\varepsilon) |S|^2$.
  We could write this in terms of convolution as $1_S * 1_S * 1_S(I) \geq (1-10\varepsilon)\mu^2|S|^2$.
  Letting $\widehat{G}$ be the group of multiplicative characters $G \to \C^\times$, by Fourier inversion, we therefore have $\mathbb{E}_{\chi \in \widehat{G}} \widehat{1_S}(\chi)^3 \geq (1-10\varepsilon)|S|^2$.
  Note also that, since $1_S$ is the indicator function of a symmetric set of size $|S|$, each Fourier coefficient $\widehat{1_S}(\chi)$ is a real number between $-|S|$ and $|S|$, and by Plancharel's formula we have $\mathbb{E}_{\chi \in \widehat{G}} \widehat{1_S}(\chi)^2 = |S|$.\\
  
  Letting $\Triv\colon G \to \C^\times$, $x \mapsto 1$ be the trivial character, we have $\widehat{1_S}(\Triv) = |S|$. So,
  \[(1-10\varepsilon)|S|^2 - \frac{|S|^3}{|G|} \leq \E_{\chi \neq \Triv} \widehat{1_S}(\chi)^3 \leq \left(\max_{\chi \neq \Triv} \widehat{1_S}(\chi) \right)\left( \E_{\chi \neq \Triv} \widehat{1_S}(\chi)^2 \right) = \left(\max_{\chi \neq \Triv} \widehat{1_S}(\chi) \right)\left( |S| - \frac{|S|^2}{|G|} \right).\]

  Rearranging, this gives
  \[\max_{\chi \neq \Triv} \widehat{1_S}(\chi) \geq |S| \left(1 - \frac{10\varepsilon}{1 - \frac{|S|}{|G|}} \right) > .999 |S|. \]

  Letting $\chi^*$ be a character with $\widehat{1_S}(\chi^*) > .999|S|$, observe that we have 
  \[\widehat{1_S}(\chi^*) = \Re(\widehat{1_S}(\chi^*)) = \Re\left(\sum_{g \in G} 1_S(g)\chi^*(g)\right) = \Re\left(\sum_{x \in S} \chi^*(x)\right).\]
  Since we have $\Re(\chi^*(x)) \leq 1$ for all $x \in S$, by Markov's inequality there must be at least $.999|S|$ values of $x \in S$ with $\Re(\chi^*(x)) > 0$.
\end{proof}

Now, using \cref{lem:buff-character}, we show that most of $S$ lies within some nontrivial subgroup of $G$. 

\begin{lemma}\label{lem:pass-to-subgroup}
  Let $G$ be an abelian group, and $S \subseteq G$ be a symmetric subset.
  Call an element $x \in S$ \defn{good} if $\Pr_{y \in S}[x + y \in S] > 1/2$. 
  If $|S| < (1-20000\varepsilon)|G|$, and $\Pr_{a,b \in S}[a+b \in S] \geq 1 - 10\varepsilon$, then  there exists some nontrivial subgroup $H \leq G$ such that $H \cap S$ contains all of the good elements of $S$. 
\end{lemma}
\begin{proof}
  By \cref{lem:buff-character}, there exists some character $\chi^*$ such that at least $.999|S|$ values of $x \in S$ have $\Re(\chi^*(x)) > 0$.
  For each $x \in S$, we define $\theta_x \in [-\pi, \pi)$ to be the angle such that $\chi^*(x) = e^{i\theta_x}$.
  The condition that $\Re(\chi^*(x)) > 0$ is equivalent to $\theta_x \in \left(-\frac{\pi}{2}, \frac{\pi}{2}\right)$, so we know that most elements have somewhat small angles.
  Our goal is to apply a recursive correction argument to show that in fact most elements have angle $0$, which will then correspond to many elements in a nontrivial subgroup.
  The mechanism for doing so is the following two claims.

  \begin{claim}\label{clm:shrink-range}
  For any $\theta < \frac{\pi}{2}$, and any $\delta > \sqrt{10\varepsilon}$, if at least $(1-\delta)|S|$ many elements $x$ have $\theta_x \in [-\theta, \theta]$, then at least $(1-3\delta)$ many elements $x$ have $\theta_x \in \left[-\frac{\theta}{2}, \frac{\theta}{2}\right]$.
  \end{claim}
  \begin{proof}
    Suppose for the sake of contradiction that this is not the case. Then, there are at least $2\delta |S|$ elements with $\theta_x \in \big(\frac{\theta}{2}, \theta\big] \cup \big[-\theta, -\frac{\theta}{2}\big)$.
    By symmetry of $S$, this means there are $\delta |S|$ many elements with $\theta_x \in \big(\frac{\theta}{2}, \theta\big]$.
    For any pair $x, y$ of such elements, we know $\theta_{x+y} = \theta_x + \theta_y \in (\theta, 2\theta] \subseteq (\theta, \pi)$. 
    But we know there are only at most $\delta |S|$ many elements with $\theta_x$ outside of $[-\theta, \theta]$, so by symmetry of $S$ there are at most $\frac{\delta}{2} |S|$ many elements of $S$ with $\theta_x \in (\theta, \pi)$.
    Since, for a given $x$, every $y$ yields a distinct value of $x+y$, this means that there are at least $\delta |S| \cdot \frac{\delta}{2}|S| = \frac{\delta^2|S|^2}{2}$ pairs $x, y \in S$, with $\theta_x, \theta_y \in \big(\frac{\theta}{2}, \theta\big]$, such that $x + y \not\in S$.
    By symmetry, there are also at least $\frac{\delta^2|S|^2}{2}$ such pairs with $\theta_x, \theta_y \in \big[-\theta, -\frac{\theta}{2}\big)$. 
    But now, since there are at most $3\varepsilon|S|^2$ pairs $x,y\in S$ with $x+y \not\in S$, this implies $\delta^2 \leq 10\varepsilon$.
  \end{proof}

  \begin{claim}\label{clm:grow-range}
    For any $\theta < \frac{\pi}{2}$, if at least $\frac{3|S|}{4}$ many elements $x$ have $\theta_x \in [-\theta, \theta]$, then all good elements $x$ have $\theta_x \in [-2\theta, 2\theta]$.
  \end{claim}
  \begin{proof}
    Let $x \in S$ be such that $\theta_x \not\in [-2\theta, 2\theta]$. 
    Then, for any $y \in S$ with $\theta_y \in [-\theta, \theta]$, we have $\theta_{x+y} = \theta_x + \theta_y \not\in [-\theta, \theta]$. 
    Since there are at least $\frac{3|S|}{4}$ choices for such $y$, and only at most $\frac{|S|}{4}$ choices of $x+y \in S$ with $\theta_x+\theta_y \not\in [-\theta,\theta]$, this means that $\Pr_{y\in S}[x+y \not\in S] \geq 1/2$.
  \end{proof}

  To deduce the lemma from these two claims, let $\theta \geq 0$ be the smallest value such that all good elements have $\theta_x \in [-\theta,\theta]$.
  First, we observe that $\theta < \pi/2$. This follows because we are initially given that at least $.999|S|$ elements have $\theta_x \in \left(-\frac{\pi}{2}, \frac{\pi}{2} \right)$, so \cref{clm:shrink-range} ensures at least $\frac{3}{4}|S|$ many elements have $\theta_x \in \left(-\frac{\pi}{4}, \frac{\pi}{4}\right)$, so, in turn, \cref{clm:grow-range} ensures that all good elements have $\theta_x \in \left( -\frac{\pi}{2}, \frac{\pi}{2}\right)$.\\

  Now, note that there must be at least $1 - 20\varepsilon$ good elements, since otherwise $\Pr_{a,b\in S}[a+b \in S]< 1-10\varepsilon$. 
  Since the statement of \cref{lem:pass-to-subgroup} is vacuous when $\varepsilon \geq 1/20000$, we can assume $\varepsilon < 1/20000$, in which case $1 - 20\varepsilon \geq 1 - \sqrt{10\varepsilon}$.
  Applying \cref{clm:shrink-range} twice, this implies that there are at least $(1-9\sqrt{10\varepsilon}) |S| > \frac{3|S|}{4}$ many elements $x$ with $\theta_x \in \left[-\frac{\theta}{4}, \frac{\theta}{4}\right]$. 
  So, applying \cref{clm:grow-range} implies that all good elements have $\theta_x \in \left[-\frac{\theta}{2}, \frac{\theta}{2}\right]$.
  By minimality of $\theta$, this implies $\theta = \frac{\theta}{2}$, which implies $\theta = 0$. 
  Thus, all good $x$ have $\theta_x = 0$, meaning that they lie in the kernel of $\chi^*$.
  Because the kernel of any character is a subgroup of $G$ and $\chi^*$ is not the trivial character, the kernel is not all of $G$, so we have shown that all good $x$ belong to the same nontrivial subgroup.
\end{proof}

With this machinery in place, it is not too difficult to conclude \cref{lem:subgroup-testable}.

\begin{proof}[Proof of \cref{lem:subgroup-testable}]
  Take $K = 40000$. 
  Because the statement is vacuous when $\varepsilon > \frac{1}{K}$, we will consider only $\varepsilon \leq \frac{1}{40000}$.\\

  Call an element $a \in S$ \defn{super good} if $\Pr_{b \in S}[a+b\in S] \geq 2/3$.
  Let $G' \leq G$ be the smallest subgroup containing all super good elements of $S$, and let $S' = S \cap G'$.
  Observe by Markov's inequality that there are at most $3 \varepsilon |S|$ elements of $S$ that are not super good, so $|S'| \geq (1 - 3\varepsilon)|S|$.
  It will therefore suffice for the desired statement to show that $|G'| \leq (1 + 40000 \varepsilon)|S|$.\\

  If $|S'| \geq (1-20000\varepsilon)|G'|$, then $|G'| \leq \frac{1}{1-20000\varepsilon}|S'| \leq \frac{1}{1-20000\varepsilon}|S| \leq (1+40000\varepsilon)|S|$.
  So, let us assume for the sake of contradiction that $|S'| < (1-20000\varepsilon)|G'|$.
  Note that we have 
  \[\Pr_{a,b \in S'}[a+b \in S'] =  \Pr_{a,b \in S'}[a+b \in S] \geq \Pr_{a,b \in S}[a + b \in S] - \Pr_{a\in S}[a \not\in S'] - \Pr_{b \in S}[b \not\in S'] \geq 1 - 7\varepsilon.\]
  So, we can apply \cref{lem:pass-to-subgroup} to find a nontrivial subgroup $H \leq G'$ such that all $a$ with $\Pr_{b \in S'}[a+b \in S']> 1/2$ lie in $H$.
  But now note that, for any $a \in S'$,
  \[\Pr_{b \in S'}[a + b \in S'] = \Pr_{b \in S'}[a + b \in S] \geq \Pr_{b \in S}[a + b \in S] - \Pr_{b \in S}[b \not\in S'] \geq \Pr_{b \in S}[a + b \in S] - 3\varepsilon.\]
  Thus, if $a$ is super good in $S$, then we must have $\Pr_{b \in S'}[a+b \in S']> 2/3 - 3 \varepsilon \geq 1/2$, which implies $a \in H$.
  So, all super good elements lie in $H$ --- but this contradicts our assumption that $G'$ was the minimal subgroup containing all super good elements.
\end{proof}

One might at this point ask whether \cref{lem:subgroup-testable} holds for non-abelian groups, as this would generalize \cref{thm:rc-cayley-nonexistence} to arbitrary Cayley graphs. 
We suspect that such a statement might be true, although we do not know how to prove it.
A potential approach to proving such results might be through the Balog-Szemer\'edi-Gowers theorem, which gives that, if $\Pr_{a, b, c \in S}[a + b - c \in S] \geq 1/c$, then there exists $A' \subseteq A$ with $|A' + A'| \leq K \cdot c^{O(1)} |A'|$ and $|A'| \geq c^{-O(1)}|A| / K$, where $A' + A' = \{a+b\ |\ a,b \in A'\}$~\cite{balog1994statistical,gowers1998new}.
This result was originally proved for abelian groups, but Tao has generalized it to non-abelian groups~\cite{tao2008product}.
If the statement held with the constant factor $K$ equal to $1$, this would suffice to prove \cref{lem:subgroup-testable} for general groups: the condition $\Pr_{a,b \in S}[a+b\in S] \geq (1-\eps)$ implies $\Pr_{a, b, c \in S}[a + b - c \in S] \geq 1/(1+O(\eps))$, and so we would get $A' \subseteq A$ with $|A'| \geq (1-O(\eps))|A|$ and $|A' + A'| \leq (1+O(\eps))|A|$.
Frieman has shown that, for any set $A'$ with $|A' + A'| \leq \frac{3}{2}|A'|$, it must be that $|A' + A'|$ is a coset, which in our case would mean that $|A' + A'|$ is a subgroup, so this would imply the desired result~\cite{freiman1973number}.
So, an attractive means of establishing \cref{lem:subgroup-testable} for general groups might be searching for variants of the Balog-Szemer\'edi-Gowers theorem that give good quantitative guarantees even for $c$ very close to $1$.
In the nonabelian case, Shao has shown such an ``almost-all'' version of the Balog-Szemer\'edi-Gowers theorem: given any $\eps>0$, there exists a $\delta>0$ such that $\Pr_{a, b, c \in S}[a + b - c \in S] \geq 1/(1+\delta)$ implies the existence of $A' \subseteq A$ with $|A' + A'| \leq (1+\eps) |A'|$ and $|A'| \geq (1-\eps)|A|$~\cite{shao2018almost}.
The quantitative bounds on $\delta$ in terms of $\epsilon$ come from the arithmetic removal lemma and as such are very far from the linear dependence we would need to establish \cref{lem:subgroup-testable}, however one could at least recover a quantitatively weaker version of \cref{thm:rc-cayley-nonexistence} for arbitrary Cayley graphs if one could generalize Shao's result to non-abelian groups.

%% file: lp.tex
\section{Existence of $(\vec{r}, \vec{c})$-triangle-regular graphs}\label{sec:lp}

At this point, we have given a partial characterization of values of
$r$ and $c$ that admit an $(r, c)$-triangle-regular graph. We now
study a generalization of this problem to multi-colored
graphs. Namely, our desire is to characterize which $\vec{r}$ and
$\vec{c}$ admit an $(\vec{r}, \vec{c})$-triangle-regular graph.  Note
that, when $t=1$, this is exactly the problem we have been studying
thus far.


\multicolor*


\subsection{Necessary linear programming condition}

The first observation we make is a simple linear system whose feasibility is necessary for the existence of an $(\vec{r}, \vec{c})$-triangle-regular graph.
\begin{proposition}\label{prop:lp-necessary}
    If there exists an $(\vec{r}, \vec{c})$-triangle-regular graph $G$, then there exists a solution to the following system, with variables $x_{ijk}$ for each unordered triple $i,j,k \in \{1,\dots,t\}$:
    \begin{align}
        0 &\leq x_{ijk} & \forall i,j,k \label{eq:1}\\
        \sum_{k \neq i,j} x_{ijk} + 2 x_{iij} + 2 x_{ijj} &\leq \vec{r}[i]\vec{r}[j] &\forall i \neq j \label{eq:2}\\
        \sum_{k \neq i} x_{iik} + 3x_{iii} &\leq \binom{\vec{r}[i]}{2} &\forall i \label{eq:3}\\
        \sum_{j,k \neq i} x_{ijk} + 2\sum_{j \neq i} x_{iij} + 3 x_{iii} &= \vec{c}[i]  &\forall i. \label{eq:4}
    \end{align}
\end{proposition}
\begin{proof}
    Let $n$ be the number of vertices in $G$, and for each unordered triple $i,j,k$, let $x_{ijk}$ be $\frac{1}{n}$ times the number of triangles with edge colors $i$, $j$, and $k$.
    The neighborhood of each vertex $v$ induces exactly $\vec{c}[k]$ edges of color $k$, so we know that this also holds on average, that is, $\sum_{i,j,k} [\text{\# occurences of $k$ among $i,j,k$}] \cdot x_{ijk} = \vec{c}[k]$. This corresponds to \cref{eq:4}.
    Now, note that if $\sum_{k \neq i,j} x_{ijk} + 2 x_{iij} + 2 x_{ijj} > \vec{r}[i]\vec{r}[j]$ for some $i \neq j$, then by averaging there must be some vertex $v$ such that there are more than $\vec{r}[i]\vec{r}[j]$ many edges between $v$'s $i$-neighborhood and $j$-neighborhood.
    But this is a contradiction, because those neighborhoods have sizes only $\vec{r}[i]$ and $\vec{r}[j]$, respectively.
    By the same reasoning, we cannot have $\sum_{j,k \neq i} x_{ijk} + 2\sum_{j \neq i} x_{iij} + 3 x_{iii} > \binom{\vec{r}[i]}{2}$ for any $i$.
\end{proof}

Note that feasibility of this system is \emph{not} a sufficient condition for the existence of an $(\vec{r},\vec{c})$-triangle-regular graph in general.
For instance, for $t=1$, this system is feasible for all $c \leq \binom{r}{2}$, but we know by \cref{thm:caro-et-al-rc} that there exist some such values not achievable by an $(r,c)$-triangle-regular graph.
However, we note that it \emph{is} ``approximately sufficient'' in the following sense.

\begin{proposition}\label{prop:lp-kinda-sufficient}
    Suppose, for a fixed $\vec{r}$, $\vec{c}$, there exists a solution to the system of equations in \cref{prop:lp-necessary}.
    Then, for all $i$, there exists a $\vec{c'}$ with $\floor{\frac{\vec{c}[i]}{8t^3}} \leq \vec{c'}[i] \leq \vec{c}[i]$ such that there exists an $\left(\vec{r},\vec{c'}\right)$-triangle-regular graph.
\end{proposition}
\begin{proof}
    For every $i,j,k$, we construct a graph $G_{ijk}$, and we eventually take the Cartesian product of all of these graphs to yield our $\left(\vec{r},\vec{c'}\right)$-triangle-regular graph.

    \paragraph{Case 1} $i = j = k$. In this case, let $y$ be the largest integer such that $\binom{y - 1}{2} \leq x_{iii}$, and let $G_{iii}$ be a clique of color $i$ on $\min\left(y, \frac{\vec{r}[i]}{t^2} + 1\right)$ vertices.

    \paragraph{Case 2} $i,j$, and $k$ are all distinct. Assume without loss of generality that $\vec{r}[i] \leq \vec{r}[j] \leq \vec{r}[j]$. 
    Let $y$ be the largest integer such that $2 y \cdot \frac{\vec{r}[i]}{t^2} \leq x_{ijk}$.
    We construct a graph $G_{ijk}$ with four parts, each with $\min\left(y, \frac{\vec{r}[j]}{t^2}\right)$ vertices. 
    Put complete bipartite graphs colored $j$ between two disjoint pairs of parts, and complete bipartite graphs colored $k$ between another two disjoint pairs of parts.
    Between the remaining two pairs of parts, put a $\left(\frac{\vec{r}[i]}{t^2}\right)$-regular bipartite graph colored $i$ (or a complete bipartite graph if this is larger than the number of vertices in a part).
    
    \paragraph{Case 3} $i = j \neq k$ (the other two remaining cases are symmetric). 
    Here, we give exactly the same construction as in the previous case,
    except that we halve the size of every part. \\
    
    The key facts about these constructions are:
    \begin{enumerate}
        \item Each vertex in $G_{ijk}$ is only involved in triangles with edge colors $i,j,k$, and is involved an equal number of times as each vertex position in such a triangle.
        \item The total number of triangles in $G_{ijk}$ is at least $\floor{\frac{x_{ijk}}{8t^2}}$, and at most $x_{ijk}$.
        \item Each color $i$ in each graph $G_{ijk}$ gives a regular graph with degree at most $\frac{\vec{r}[i]}{t^2}$.\\
    \end{enumerate}

    Now, consider the Cartesian product of all of these graphs. 
    By i), we have ensured that each neighborhood induces the same number of edges of color $i$.
    By ii), we know that this number is between $\sum_{j,k \neq i} \floor{\frac{x_{ijk}}{8t^2}} + 2\sum_{j \neq i} \floor{\frac{x_{iij}}{8t^2}} + 3 \floor{\frac{x_{iii}}{8t^2}} \geq \floor{\frac{\vec{c}[i]}{8t^3}}$ and $\sum_{j,k \neq i} x_{ijk} + 2\sum_{j \neq i} x_{iij} + 3 x_{iii} = \vec{c}[i]$ edges of color $i$.
    Then, by iii), we know that the degree in each color is at most $t^2 \cdot \frac{\vec{r}[t]}{t^2} = \vec{r}[t]$.
    By taking the Cartesian product of this graph with bipartite graphs in each color with smaller degree than $\vec{r}[t]$, we can make the resulting graph $\vec{r}$-regular without creating any new triangles.
    So, we have found a graph satisfying the desired conditions.
\end{proof}

\subsection{Analysis of specific flip graph problems}

\cref{prop:lp-kinda-sufficient} implies that the linear program given by \cref{prop:lp-necessary} gives some sort of approximate characterization of which $\vec{r}$, $\vec{c}$ suffice.
We suspect, however, that these linear conditions are often a much better characterization than what \cref{prop:lp-kinda-sufficient} indicates.
To capture this, we use this linear program to study two questions posed as open problems by Caro et al. \cite{flip1}: we resolve these problems by showing that this linear program (or a mild modification) gives tight results in particular settings.

Both problems concern the notion of a flip graph.

\begin{defin}\label{def:flip}
    A $t$-colored graph $G$ is called a \defn{flip graph} if, for every vertex $v$ and every $i, j \in \{1,\dots,t\}$ with $i<j$, $v$ is incident to strictly more color-$j$ edges than color-$i$ edges, but $N(v) \cup \{v\}$ induces strictly more color-$i$ edges than color-$j$ edges.
    When $G$ is $(\vec{r}, \vec{c})$-triangle-regular, this condition can be written as $\vec{r}[i] < \vec{r}[j]$ and $\vec{r}[i] + \vec{c}[i] > \vec{r}[j] + \vec{c}[j]$ for all $i < j$.
\end{defin}

Caro et al. write that this notion is interesting particularly because it exhibits a``local versus global'' phenomenon: in a flip graph, the closed neighborhood of any vertex contains counts of each color in decreasing order, while the graph as a whole contains counts of each color in increasing order.
They are particularly interested in characterizing which $\vec{r}$ allow the existence of a $(\vec{r}, \vec{c})$-triangle-regular flip graph for some $\vec{c}$, referring to such $\vec{r}$ as \defn{flip sequences}.
They give a complete characterization for $t=2$, but leave several questions open for $t>2$. 
In this section, we address two such questions, obtaining upper bounds via the system of \cref{prop:lp-necessary} and extensions, and giving lower-bound constructions explicitly.

\subsubsection{Maximizing the largest degree in a $3$-term flip sequence}

The simplest case that Caro et al. are unable to characterize is the $3$-color case.
Here, they pose the following question.

\begin{question}[\cite{flip1}]
    Given $a_1 \in \N$, for which values $a_3 \in \N$ do there exist $a_2 \in \N$ such that $(a_1, a_2, a_3)$ is a flip sequence?
\end{question}

Caro et al. show that $a_3 \approx \frac{a_1^2}{2}$ is possible and that $a_3 > 2a_1^2$ is not. 
We improve both the upper and lower bounds. 

\begin{theorem}
    For every $a_1 \in \N$, there exists an $a_2 \in \N$ such that $\left(a_1, a_2, \left(a_1 - 2\ceil{\sqrt{a_1}}\right)^2 \right)$ is a flip sequence, and every flip sequence $\left(a_1, a_2, a_3\right)$ has $a_3 \leq \frac{5 a_1^2}{4}$.
\end{theorem}
\begin{proof}
    To show that $a_3 = \left(a_1 - 2\ceil{\sqrt{a_1}}\right)^2$ is possible, we give a construction.
    Take two cliques of color $1$ on $(a_1 - 2\ceil{\sqrt{a_1}} + 1)$ vertices and add a complete bipartite graph of color $2$ between them.
    What we have now is a graph with degree $(a_1 -2\ceil{\sqrt{a_1}})$ and $(a_1 - 2\ceil{\sqrt{a_1}} + 1)$ in the first and second colors, respectively, and closed-neighborhood edge counts of
    $\binom{a_1-2\ceil{\sqrt{a_1}} + 1}{2} + \binom{a_1 - 2\ceil{\sqrt{a_1}}}{2} = (a_1-2\ceil{\sqrt{a_1}})^2$ and $(a_1-2\ceil{\sqrt{a_1}} + 1)^2$ in the first and second colors respectively.
    Now, take the Cartesian product of this graph with a color-1 clique on $2\ceil{\sqrt{a_1}} + 1$ vertices,
    and a color-2 complete bipartite graph on $4\ceil{\sqrt{a_1}}$ vertices. The resulting graph has degrees $a_1 < a_1 + 1$ and closed neighborhood edge-counts $(a_1 - 2\ceil{\sqrt{a_1}})^2 + \binom{2\ceil{\sqrt{a_1}}+1}{2} > (a_1 - 2\ceil{\sqrt{a_1}} + 1)^2 + 2\ceil{\sqrt{a_1}}$ in the two colors.
    Finally, take the Cartesian product of this resulting graph with a color-$3$ complete bipartite graph on $2(a_1 - 2\ceil{\sqrt{a_1}})^2$ vertices, and the graph obtained is a flip graph.\\

    To get an upper bound on $a_3$, we can use \cref{prop:lp-necessary}.
    Taking the linear system from the proposition with $\vec{r}$ and $\vec{c}$ as variables, and with the additional constraints that $\vec{r}[1] + \vec{c}[1] > \vec{r}[2] + \vec{c}[2] > \vec{r}[3] + \vec{c}[3]$, that $\vec{r}[3] > \vec{r}[2] > \vec{r}[1]$, and that $\vec{r}[3] > \frac{3\vec{r}[1]^2}{2}$, we obtain a small polynomial system that can be explicitly checked to be unsatisfiable.
    However, if we require only that $\vec{r}[3] > \frac{5\vec{r}[1]^2}{4}$ as opposed to $\vec{r}[3] > \frac{3\vec{r}[1]^2}{2}$, that system \emph{is} satisfiable.
    So, in order to get the stated bound, we need to derive some additional constraints.
    It turns out to be sufficient to add the following nonlinear constraint
    \begin{align*}
        \vec{r}[i] \cdot \sum_{j, k, \ell \neq i} x_{jk\ell} \geq \binom{\vec{r}[i]}{3} \cdot \frac{\sum_{j \neq i} x_{iij}}{\binom{\vec{r}[i]}{2}} \cdot \left(\frac{2\sum_{j \neq i} x_{iij}}{\binom{\vec{r}[i]}{2}} -1 \right)& \qquad \forall i.
    \end{align*}

    The fact that the polynomial system becomes unsatisfiable with this constraint can be verified using a computer algebra system.
    In \cref{lem:trianglecount-newconstraint}, we show that the triangle counts of any $\left(\vec{r},\vec{c}\right)$-triangle regular graph must satisfy this constraint, so this proves that we cannot take $\vec{r}[3] > \frac{5\vec{r}^2}{4}$.
\end{proof}

\begin{lemma}\label{lem:trianglecount-newconstraint}
    Let $G$ be an $\left(\vec{r}, \vec{c}\right)$-triangle regular graph, and let $x_{ijk}$ denote $\frac{1}{n}$ times the total number of triangles with edge colors $i,j$, and $k$.
    Then, for any $i$, we have 
    \begin{align*}
        \vec{r}[i] \cdot \sum_{j, k, \ell \neq i} x_{jk\ell} \geq \binom{\vec{r}[i]}{3} \cdot \frac{\sum_{j \neq i} x_{iij}}{\binom{\vec{r}[i]}{2}} \cdot \left(\frac{2\sum_{j \neq i} x_{iij}}{\binom{\vec{r}[i]}{2}} -1 \right)& \qquad \forall i.
    \end{align*}
\end{lemma}

\begin{proof}
    Note that $n \cdot \sum_{j, k, \ell \neq i} x_{jk\ell}$ is the total number of triangles all of whose edges are colors other than $i$.
    Since every vertex has exactly $\vec{r}[i]$ incident edges of color $i$, we know that the vertices of any such triangle have at most $\vec{r}[i]$ common $i$-neighbors.
    So, we can lower bound the total number of non-$i$ triangles by $\frac{1}{\vec{r}[i]} \cdot \sum_{v} [\text{\# non-$i$ triangles in $v$'s $i$-neighborhood}]$.\\

    Now, we can lower bound those counts by using triangle supersaturation. 
    It is known that any $N$-vertex graph of edge density $\rho$ must contain at least $\binom{N}{3} \cdot \rho(2\rho - 1)$ triangles~\cite{goodman1959sets}.
    (In fact, tight lower bounds are known by a result of Razborov~\cite{razborov2008minimal} --- however, this simpler bound suffices for our purposes.)
    So, if we let $y_v$ denote the number of non-$i$ edges in $v$'s neighborhood, we can write 
    \[\sum_{v} [\text{\# non-$i$ triangles in $v$'s $i$-neighborhood}] \geq \sum_{v} \binom{\vec{r}[i]}{3} \cdot \min\left(0, \left(\frac{y_v}{\binom{\vec{r}}{2}}\right) \cdot \left(\frac{2y_v}{\binom{\vec{r}}{2}} - 1\right)\right).\]
    As $z \mapsto \min\left(0, \left(\frac{z}{\binom{\vec{r}}{2}}\right) \cdot \left(\frac{2z}{\binom{\vec{r}}{2}} - 1\right)\right)$ is a convex function in $z$, we can further write
    \begin{align*} 
        \frac{1}{n}\sum_{v} \binom{\vec{r}[i]}{3} \cdot \min\left(0, \left(\frac{y_v}{\binom{\vec{r}}{2}}\right) \cdot \left(\frac{2y_v}{\binom{\vec{r}}{2}} - 1\right)\right) &\geq \binom{\vec{r}[i]}{3} \cdot \min\left(0, \left(\frac{\sum_v y_v}{n \cdot \binom{\vec{r}}{2}}\right) \cdot \left(\frac{2\sum_v y_v}{n \cdot \binom{\vec{r}}{2}} - 1\right)\right)\\
        &= \binom{\vec{r}[i]}{3} \cdot \left(\frac{\sum_{j \neq i} x_{iij}}{\binom{\vec{r}}{2}}\right) \cdot \left(\frac{2\sum_{j \neq i} x_{iij}}{\binom{\vec{r}}{2}} - 1\right).
    \end{align*}
    Combining these inequalities yields the desired statement.
\end{proof}

\subsubsection{Unbounded flip sequences for $t > 3$}

For flip sequences longer than three terms, however, Caro et al. show that $a_t$ can no longer be bounded in terms of $a_1$.
Specificially, they show that for any $t$, there exists some value $a_1$ such that for all $B$ there exists a flip sequence $(a_1, a_2, \dots, a_t)$ with $a_t > B$.
In other words, as long as $a_1$ is larger than some fixed constant depending only on $t$, $a_t$ can be made arbitrarily large.
In their proof of this result, the resulting constant $a_1$ grows with $t$; they ask what the true rate of growth is.

\begin{question}[\cite{flip1}]
    For a given $t$, what is the minimum constant $a_1$ such that there exist flip sequences $(a_1, \dots, a_t)$ with $a_t$ arbitrarily large?
\end{question}

In the following, we demonstrate by construction that in fact $a_1$ can be taken to $1$ as $t$ grows large.
Furthermore, we show using the system of \cref{prop:lp-necessary} that this construction is tight.

\begin{theorem}
    We have the following:
    \begin{itemize}
        \item There exists a flip sequence $(3, a_2, a_3, a_4)$ for arbitrarily large $a_4$, but when $a_1 \leq 2$, then $a_4$ cannot be made arbitrarily large.
        \item There exists a flip sequence $(2, a_2, a_3, a_4, a_5)$ for arbitrarily large $a_5$, but when $a_1 = 1$, then $a_5$ cannot be made arbitrarily large.
        \item There exists a flip sequence $(1, a_2, a_3, a_4, a_5, a_6)$ for arbitrarily large $a_6$.
    \end{itemize}
\end{theorem}
\begin{proof}
    In all three cases, the construction is as follows: take $20000B(t-1)$ cliques of size $8-t$ in color $1$. 
    Then, place a complete bipartite graph between the first $10000B(t-1)$ of those cliques and the other $10000B(t-1)$ of them.
    Color this complete bipartite graph with colors $a_2, \dots, a_t$ in such a way as to ensure that each vertex has degree $10000B(8-t)$ in each of those colors.
    (For instance, decompose the complete bipartite graphs into matchings and color $10000B$ of the matchings color $i$.)
    Let $G$ be the graph that we have constructed at this point. 
    Then, take the Cartesian product of $G$ with a small flip graph on colors $(a_2, \dots, a_t)$.
    (Caro et al. give small constructions~\cite{flip1} --- we certainly know, for instance, that there exists some flip graph with fewer than $100$ vertices whenever $t \leq 6$.)\\

    To see that this construction works, consider the closed neighborhood $N[v]$ of any vertex $v$ in $G$.
    We have that $N[v]$ consists of its own clique, plus all $10000B$ cliques on the other side.
    So, $N[v]$ contains $(10000B + 1)(t-1)\binom{8 - t}{2}$ edges of color $1$, and $10000B(8-t)^2$ edges of each other color.
    Because $(t-1)\binom{8 - t}{2} > (8-t)^2$ for all $t \in \{4,5,6\}$, $N[v]$ contains far more edges of color $1$ than it does edges of the other colors.
    Since all other colors have exactly the same degree and closed-neighborhood edge count for every vertex, taking the Cartesian product of $G$ with a small flip graph on these colors $(a_2, \ldots, a_t)$, we obtain a flip graph.\\

    To show that this construction is tight, we can take for each $t \in \{4,5,6\}$ the system of \cref{prop:lp-necessary}, along with the conditions that
    \begin{align}
      0 &\leq \vec{r}[1] \leq 7 - t \\
      \vec{r}[t] &> 1000 \\
      \vec{r}[1]+1 &\leq \vec{r}[2];~\dots~; \vec{r}[t-1]+1 \leq \vec{r}[t] \\
      \vec{r}[1] + \vec{c}[1] &\geq \vec{r}[2] + \vec{c}[2] + 1; \vec{r}[2] + \vec{c}[2] \geq \vec{r}[3] + \vec{c}[3] + 1;~\dots~;\vec{r}[t-1] + \vec{c}[t-1] \geq \vec{r}[t] + \vec{c}[t] + 1.
    \end{align}

    Each of these systems can be explicitly verified to be unsatisfiable.
\end{proof}

%% file: hardness.tex
\section{Hardness of $(\vec{r}, \vec{c})$-triangle-regular colorings and flip colorings}\label{sec:hard}

In \cref{sec:lp}, we addressed questions of the form ``given $\vec{r}$ and $\vec{c}$, does there exist an $(\vec{r}, \vec{c})$-triangle-regular graph?''.
Another interesting, related question is: given a \emph{particular} graph $G$, do there exist $\vec{r}$ and $\vec{c}$ such that $G$ admits an $(\vec{r}, \vec{c})$-triangle-regular coloring?
In this section, we demonstrate that such questions are $\NP$ hard, in general. 
We also show by similar techniques that determining whether a graph admits a \emph{flip coloring}, as defined in the previous section, is $\NP$ hard --- this answers a question of Caro et al.~\cite{flip1}.

\subsection{$\NP$-completeness of finding a $(\vec{r}, \vec{c})$-triangle-regular coloring}

We first demonstrate the following theorem.
\rchard

We reduce from the $\NP$-complete problem \emph{Positive 1-in-3-SAT-E4}.

\begin{restatable}[Positive 1-in-3-SAT-E4]{prob}{1in3sat}\label{prob:1in3sat}
  Given a boolean formula $\varphi$ in 3-CNF (conjunctive normal form,
  where each clause contains three literals) such that (1) $\varphi$
  contains no negations, (2) each clause contains three distinct
  variables, and (3) each variable appears in exactly four clauses,
  determine whether $\varphi$ has an truth assignment such that every
  clause contains exactly one true variable, that is, a \defn{1-in-3
  assignment}.
\end{restatable}

$\NP$-hardness of this problem has been shown by Porschen, Schmidt, and Speckenmeyer~\cite{porschen2010complexity,schmidt2010computational}.

\begin{proof}[Proof of \cref{thm:rchard}]
  Containment in $\NP$ is immediate, since a coloring can be easily checked.
  We will specifically show hardness for the case where $\vec{r} = (1,6)$ and $\vec{c} = (1,2)$.
  That is, the input is a $(7,3)$-triangle-regular graph, and it is a yes instance if its edges can be colored red and blue, such that every vertex has exactly one red neighbor, and exactly one red edge in its neighborhood.
  To show hardness, we construct, given a Positive 1-in-3-SAT-E4 instance $\varphi$, a graph $G$ that is $(\vec{r}, \vec{c})$-triangle-regular-colorable if and only if $\varphi$ has a 1-in-3 assignment.\\

  Towards this construction, let us first suspend disbelief and assume that we are capable of also adding ``dangling'' edges with only one endpoint.
  At the end of the description, we will pair off those edges and get an actual graph (without creating any triangles involving those edges).
  Using these dangling edges, we can design the following variable gadget.
  We think of this gadget as having $8$ ``internal vertices,'' and $4$ ``attachment vertices'' --- the internal vertices have degree $7$ already in this gadget, whereas the attachment vertices only have degree $2$, and must be combined with other components of the overall construction.
  In order to satisfy the $(\vec{r}, \vec{c})$-triangle-regular conditions for all $8$ internal vertices, one can readily verify that the gadget must be assigned a rotation of one of the four colorings shown in \cref{fig:rc-vargadget}, which we refer to as $\texttt{TRUE}$, $\texttt{FALSE}_1$, $\texttt{FALSE}_2$, and $\texttt{FALSE}_3$, respectively.\\
  
  \begin{figure}[htp]
    \centering
    \includegraphics[width=.7\linewidth]{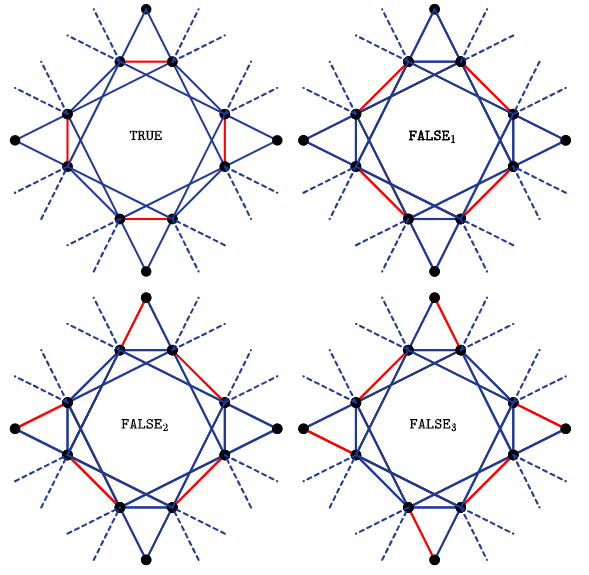}
    \caption{A variable gadget, shown in each of the four valid colorings. Dashed edges are ``dangling.'' Observe that in all of these settings, every internal vertex has $6$ blue neighbors and $1$ red neighbor, and has a neighborhood consisting of $2$ blue edges and $1$ red edge.}
    \label{fig:rc-vargadget}
  \end{figure}

  Our clause gadget will be simple: merge together $3$ attachment vertices, one from each variable gadget corresponding to a variable in the clause.
  Then, add a ``dangling'' edge. This gadget is shown in \cref{fig:rc-clausegadget}.
  Observe that, assuming that each of the vertex gadgets is given a valid assignment, such a clause vertex satisfies the conditions for an $(\vec{r}, \vec{c})$-triangle-regular graph if and only if exactly one of the three incident variable gadgets is assigned \texttt{TRUE}.
  Thus, regardless of how the dangling edges are paired up, if this graph has a valid $(\vec{r},\vec{c})$ coloring, we can obtain a 1-in-3-SAT assignment to $\varphi$ by setting every variable whose gadget is in configuration $\texttt{FALSE}_1$, $\texttt{FALSE}_2$, or $\texttt{FALSE}_3$ to be $\texttt{FALSE}$, and every variable whose gadget is set $\texttt{TRUE}$ to $\texttt{TRUE}$.\\

  \begin{figure}[htp]
    \centering
    \includegraphics[width=.8\linewidth]{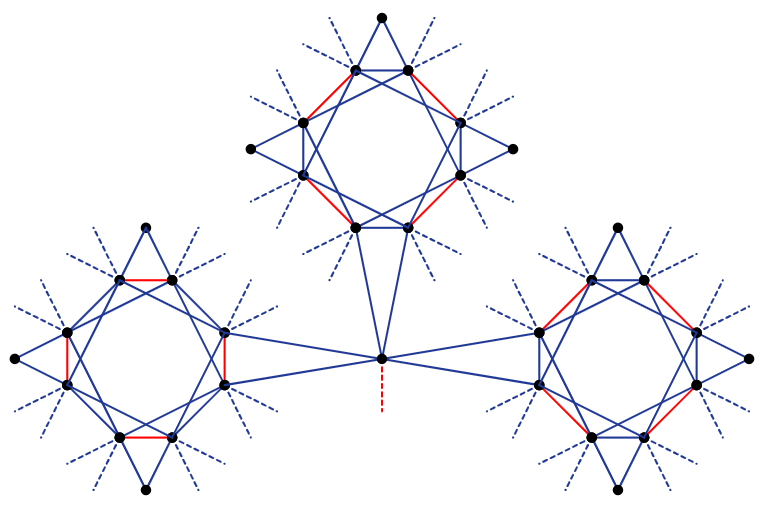}
    \caption{A clause gadget. 
    Observe that a \emph{necessary} condition for the gadget to be satisfied is that exactly one of the constituent variable gadgets is set \texttt{TRUE}, while a \emph{sufficient} condition is that the variable gadgets are set \texttt{TRUE}, $\texttt{FALSE}_1$, $\texttt{FALSE}_1$ and the dangling edge is red.
    \label{fig:rc-clausegadget}}
  \end{figure}

  It suffices now to show that the dangling edges can be paired up so that the reverse implication holds.  
  One straightforward way of doing so is to take $32$ copies of the entire construction, arrange them into two groups of $16$, and send each dangling edge from a vertex gadget to a different copy of that vertex gadget in the other group.
  This ensures that all edges dangling from vertex gadgets are paired up, without creating any new triangles.
  Each dangling edge from a clause gadget can similarly be paired with one of the other copies of that clause gadget.
  Now, given a 1-in-3-SAT assignment to $\varphi$, observe that we can obtain a valid $(\vec{r},\vec{c})$-coloring of this graph by setting all variable gadgets corresponding to true variables \texttt{TRUE}, setting all variable gadgets corresponding to false variables $\texttt{FALSE}_1$, coloring all dangling edges between variable gadgets blue, and coloring all dangling edges between clause gadgets red.

\end{proof}

\subsection{$\NP$-completeness of finding a flip coloring}

In this section, we consider instead, the notion of a \defn{flip graph}. 
Recall from \cref{def:flip} that a $2$-coloring of a graph is a \defn{flip coloring} if and only if every vertex has red degree strictly smaller than blue degree, but strictly more red than blue edges in the closed neighborhood.
In \cite{flip1}, Caro et al. pose the question of whether it is algorithmically tractable to determine if a given graph admits a flip coloring.
We show that, assuming $\P \neq \NP$, the answer is no.

\begin{theorem}\label{thm:fliphard}
    Let 2-FLIP denote the problem of determining, given a graph $G$, whether $G$ admits a $2$-color flip-coloring.
    2-FLIP is $\NP$-complete.
\end{theorem}

Our overall approach to the proof of \cref{thm:fliphard} is similar to that of \cref{thm:rchard}, but the scenario is meaningfully different.
On the one hand, we have more flexibility in design, in that the graph we produce in the reduction need not be $(r,c)$-triangle-regular.
On the other hand, the flip condition requires care to determine that given vertices have close to balanced colors in both degree and neighborhood count --- we are no longer be able to have so small a fraction of red edges as we did in the construction of \cref{thm:rchard}.
We will start from a different 3-SAT variant.

\begin{restatable}[Positive NAE-3-SAT-E4]{prob}{nae3sat}\label{prob:nae3sat}
Given a boolean formula $\varphi$ in 3-CNF such that (1)
$\varphi$ contains no negations, (2) each clause contains three
distinct variables, and (3) each variable appears in exactly four
clauses, determine whether $\varphi$ has a \defn{not-all-equals assignment}, that is, an assignment such that
there is no clause in $\varphi$ where all three variables are assigned
the same truth value.
\end{restatable}

In a recent paper, Darmann and D\"{o}cker show that Positive
NAE-3-SAT-E4 is NP-complete~\cite{darmann2020simple}. We reduce from Positive
NAE-3-SAT-E4 to 2-FLIP. 

\begin{proof}[Proof of Theorem~\ref{thm:fliphard}]

It is again clear that 2-FLIP is in $\NP$, as a $2$-coloring can be easily verified. To show that
2-FLIP is NP-hard, we show how to convert an instance $\varphi$ of Positive NAE-3-SAT-E4 into a graph $G$, such that $G$ has a flip coloring if and only if $\varphi$ has a not-all-equal assignment.\\

As in the proof of \cref{thm:rchard}, we first construct gadgets with dangling edges and attachment points, and then show how to assemble them together.
Our variable gadget will function similarly --- i.e., its core will consist of an $8$-cycle of edges that are forced to alternate between red and blue.
In order for the flip condition to force this to happen, though, we need to glue on some additional machinery.
Specifically, we end up attaching several copies of the ``auxiliary gadget'' shown in \cref{fig:auxiliary-gadget}, gluing along the vertex denoted $A$.

\begin{figure}[htp]
  \centering
  \includegraphics[width=.4\linewidth]{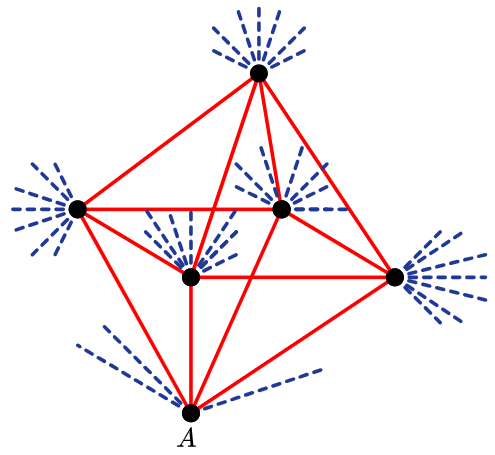}
  \caption{An auxiliary gadget, which we will glue copies of along $A$ to several vertices in the vertex gadget.}
  \label{fig:auxiliary-gadget}
\end{figure}

Observe that, as long as the dangling edges are all colored blue (which we will force to happen in the final construction), the flip condition is satisfied at all of the non-$A$ vertices if and only if the other edges are all colored red.
If they are  all red, then each of these vertices has $4$ red neighbors, $7$ blue neighbors, $8$ red edges in the closed neighborhood, and $7$ blue edges.
However, if even one edge is blue, then some closed neighborhood will have more blue than red. Now, using this, our variable gadget is shown in \cref{fig:flip-vargadget-unassigned}.

\begin{figure}[htp]
  \centering
  \includegraphics[width=.35\linewidth]{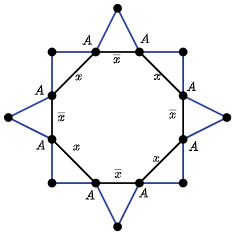}
  \caption{A variable gadget, without an assignment. Note that we assume that all edges shown other than the central $8$-cycle have been forced to be blue by the clause gadgets. We think of the edges around the cycle as alternatingly representing the variable and its negation.}
  \label{fig:flip-vargadget-unassigned}
\end{figure}

This gadget has $8$ ``internal vertices'' arranged in an $8$ cycle, each of which is attached to a copy of the auxiliary gadget.
Each edge in this $8$ cycle is used in some clause gadget, which means that it will be involved in some triangle outside the variable gadget --- by construction of the clause gadgets, the other two edges in each of those triangles will be forced to be blue.
We think of the edges around the $8$ cycle as alternating between representing the variable and its negation.
This is because the only two valid flip assignments to this gadget are as follows, which we think of as representing \texttt{TRUE} and \texttt{FALSE}, respectively:

\begin{figure}[htp]
  \centering
  \includegraphics[width=.8\linewidth]{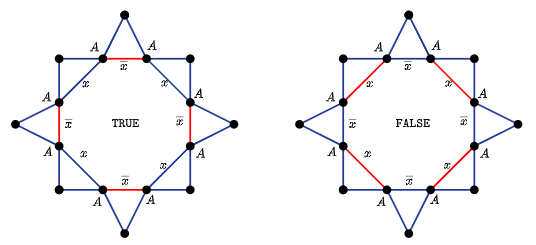}
  \caption{The two possible assignments of the variable gadget. Note that in the \texttt{TRUE} case, all edges corresponding to the variable are blue, while those corresponding to its negation are red, and vice-versa in the \texttt{FALSE} case.}
  \label{fig:flip-vargadget-unassigned}
\end{figure}

To see that these are the only two valid options, observe that, aside from the edges in the $8$-cycle, each internal vertex has blue degree $5$, red degree $4$, and blue and red closed neighborhood counts of $6$ and $8$, respectively. 
So, if both incident edges in the $8$-cycle are colored red, then the degree constraint is violated, whereas if both incident edges are blue, then the neighborhood constraint is violated. 
Thus, the valid flip colorings are precisely those that alternate between red and blue around the cycle.\\

We now discuss how to build clause gadgets on top of these variables. The structure of a clause gadget is shown in \cref{fig:flip-clausegadget-unassigned}.

\begin{figure}[htp]
  \centering
  \includegraphics[width=.7\linewidth]{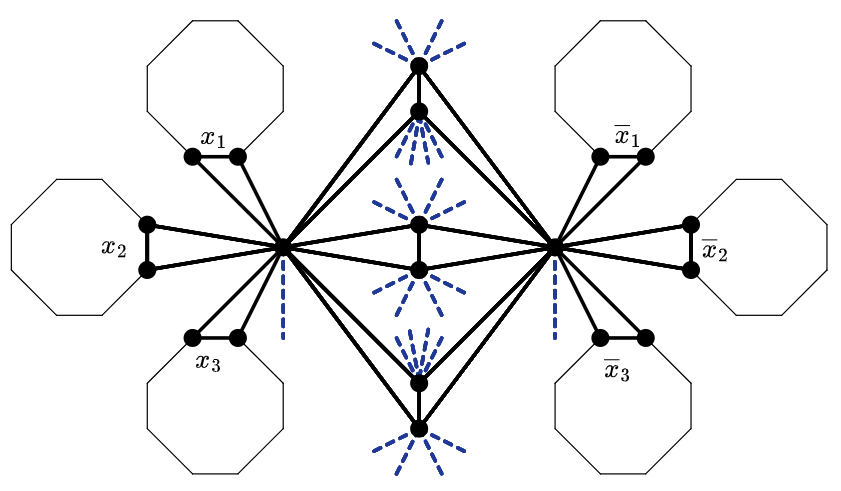}
  \caption{A gadget representing a clause on variables $x_1$, $x_2$, and $x_3$. Octagons represent the corresponding variable gadgets. Note that, while we have drawn distinct octagons for $x_i$ and $\overline{x}_i$, in reality these are two edges in the same variable gadget.}
  \label{fig:flip-clausegadget-unassigned}
\end{figure}

As with the auxiliary gadget, we assume that the dangling edges have been forced to be blue.
Given this, the neighborhood and degree constraints ensure the remaining edges must be colored as in \cref{fig:flip-clausegadget-forced}.

\begin{figure}[htp]
  \centering
  \includegraphics[width=.7\linewidth]{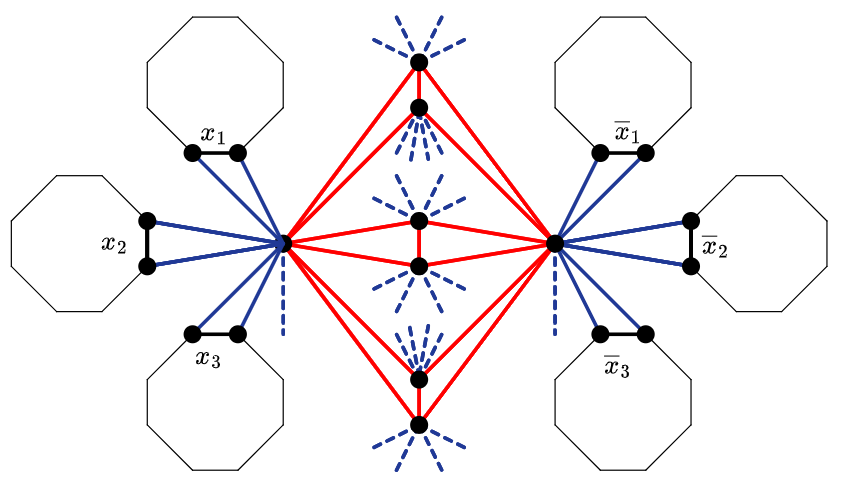}
  \caption{The colors of the internal edges of a clause vertex, as forced by flip constraints.}
  \label{fig:flip-clausegadget-forced}
\end{figure}

This leaves only the edges from the variable gadgets unforced.
Now, consider the leftmost vertex of the clause gadget.
From the forced edges, this vertex has $9$ red edges in its closed neighborhood, and $6$ blue edges.
So, its neighborhood constraint fails to be satisfied if and only if all three of the edges in its neighborhood from the variable gadgets are blue --- which happens if and only if all of $x_1$, $x_2$, and $x_3$ are set \texttt{TRUE}.
Similarly, the rightmost vertex will have its neighborhood constraint fail if and only if all of $x_1$, $x_2$, and $x_3$ are set \texttt{FALSE}.
So, this clause gadget has a valid flip coloring if and only if the assignments to its incident variable gadgets are not all equal.\\

This gives the desired reduction: if $\varphi$ admits a not-all-equal assignment, we can construct a flip coloring by setting each variable gadget to the configuration corresponding to its variable's assignment, and thereby satify all of the clause gadgets --- and likewise, given a flip coloring of the graph, since \texttt{TRUE} and \texttt{FALSE} are the only valid assignments to each variable gadget, we can read off a not-all-equal assignment to $\varphi$.\\

The only piece of the construction still remaining to be addressed is the dangling edges: several of these gadgets involved dangling edges which we asserted had been forced to be blue.
Fortunately, we can easily produce such edges using the gadget in \cref{fig:flip-blue-dangler-generator} (a $4$-clique with $4$ dangling edges from each vertex):

\begin{figure}[h]
  \centering
  \includegraphics[width=.3\linewidth]{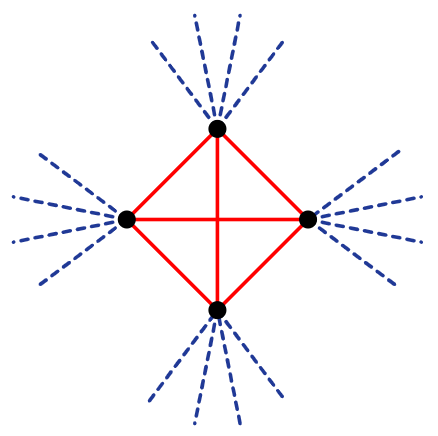}
  \caption{A blue dangler generator, shown in its only valid coloring.}
  \label{fig:flip-blue-dangler-generator}
\end{figure}

Observe the only valid flip-coloring of the this gadget is the one shown, with the clique colored red and the dangling edges colored blue.
So, to finish the construction of our graph, we can create many copies of this blue dangler generator, and pair off the dangling edges from all of the other gadgets in our construction with these gadgets' dangling edges.
(To handle modularity constraints, and to ensure that no new triangles are created, we can actually make $16$ copies of the rest of the construction, and then have each each blue dangler generator send $1$ of its $16$ edges to each of those copies.)
This gives an actual graph with no dangling edges, such that the edges our gadgets needed to have forced to be blue are all indeed forced to be blue.
\end{proof}

%% file: citations.bib
@article{flip1,
  title={Flip colouring of graphs},
  author={Caro, Yair and Lauri, Josef and Mifsud, Xandru and Yuster, Raphael and Zarb, Christina},
  journal={Graphs and Combinatorics},
  volume={40},
  number={6},
  pages={106},
  year={2024},
  publisher={Springer}
}

@article{rc-constant,
    title={On (r, c)-constant, planar and circulant graphs.},
  author={Caro, Yair and Mifsud, Xandru},
  journal={Discussiones Mathematicae: Graph Theory},
  volume={45},
  number={2},
  year={2025}
}

@article{flip2,
  title={Flip colouring of graphs II},
  author={Mifsud, Xandru},
  journal={arXiv:2401.02315},
  year={2024}
}

@article{brown1975graphs,
  title={On graphs with a constant link, II},
  author={Brown, Morton and Connelly, Robert},
  journal={Discrete Mathematics},
  volume={11},
  number={3},
  pages={199--232},
  year={1975},
  publisher={Elsevier}
}

@article{blass1980trees,
  title={Which trees are link graphs?},
  author={Blass, Andreas and Harary, Frank and Miller, Zevi},
  journal={Journal of Combinatorial Theory, Series B},
  volume={29},
  number={3},
  pages={277--292},
  year={1980},
  publisher={Elsevier}
}

@article{hall1985graphs,
  title={Graphs with constant link and small degree or order},
  author={Hall, JI},
  journal={Journal of {G}raph {T}heory},
  volume={9},
  number={3},
  pages={419--444},
  year={1985},
  publisher={Wiley Online Library}
}

@article{bulitko1973graphs,
  title={On graphs with given vertex-neighbourhoods},
  author={Bulitko, VK},
  journal={Trudy Mat. Inst. Steklov},
  volume={133},
  pages={78--94},
  year={1973}
}

@inproceedings{samotij2016number,
  title={The number of additive triples in subsets of abelian groups},
  author={Samotij, Wojciech and Sudakov, Benny},
  booktitle={Mathematical Proceedings of the Cambridge Philosophical Society},
  volume={160},
  number={3},
  pages={495--512},
  year={2016},
  organization={Cambridge University Press}
}

@article{huczynska2024additive,
  title={Additive triples in groups of odd prime order},
  author={Huczynska, Sophie and Jedwab, Jonathan and Johnson, Laura},
  journal={arXiv:2405.04638},
  year={2024}
}

@book{o2014analysis,
  title={Analysis of {B}oolean {F}unctions},
  author={O'Donnell, Ryan},
  year={2014},
  publisher={Cambridge University Press}
}

@article{nair1994triangles,
  title={About triangles in a graph and its complement},
  author={Nair, B Radhakrishnan and Vijayakumar, Ambat},
  journal={Discrete {M}athematics},
  volume={131},
  number={1-3},
  pages={205--210},
  year={1994},
  publisher={Elsevier}
}

@article{nair1996strongly,
  title={Strongly edge triangle regular graphs and a conjecture of Kotzig},
  author={Nair, B Radhakrishnan and Vijayakumar, Ambat},
  journal={Discrete Mathematics},
  volume={158},
  number={1-3},
  pages={201--209},
  year={1996},
  publisher={Elsevier}
}

@article{darmann2020simple,
  title={On a simple hard variant of Not-All-Equal 3-Sat},
  author={Darmann, Andreas and D{\"o}cker, Janosch},
  journal={Theoretical Computer Science},
  volume={815},
  pages={147--152},
  year={2020},
  publisher={Elsevier}
}

@mastersthesis{mifsud2024local,
  title={Local v. global majority: an edge colouring approach},
  author={Mifsud, Xandru},
  year={2024},
  school={University of Malta}
}

@article{goodman1959sets,
  title={On sets of acquaintances and strangers at any party},
  author={Goodman, Adolph W},
  journal={The American Mathematical Monthly},
  volume={66},
  number={9},
  pages={778--783},
  year={1959},
  publisher={Taylor \& Francis}
}

@article{razborov2008minimal,
  title={On the minimal density of triangles in graphs},
  author={Razborov, Alexander A},
  journal={Combinatorics, Probability and Computing},
  volume={17},
  number={4},
  pages={603--618},
  year={2008},
  publisher={Cambridge University Press}
}

@InProceedings{porschen2010complexity,
author="Porschen, Stefan
and Schmidt, Tatjana
and Speckenmeyer, Ewald",
editor="Strichman, Ofer
and Szeider, Stefan",
title="Complexity results for linear XSAT-problems ",
booktitle="Theory and Applications of Satisfiability Testing -- SAT 2010",
year="2010",
publisher="Springer Berlin Heidelberg",
address="Berlin, Heidelberg",
pages="251--263"
}

@phdthesis{schmidt2010computational,
  title={Computational complexity of SAT, XSAT and NAE-SAT for linear and mixed Horn CNF formulas},
  author={Schmidt, Tatjana},
  year={2010},
  school={Universit{\"a}t zu K{\"o}ln}
}

@article{berikkyzy2024triangle,
  title={Triangle-degree and triangle-distinct graphs},
  author={Berikkyzy, Zhanar and Bjorkman, Beth and Blake, Heather Smith and Jahanbekam, Sogol and Keough, Lauren and Moss, Kevin and Rorabaugh, Danny and Shan, Songling},
  journal={Discrete Mathematics},
  volume={347},
  number={1},
  pages={113695},
  year={2024},
  publisher={Elsevier}
}

@article{stevanovic2024regular,
  title={On regular triangle-distinct graphs},
  author={Stevanovi{\'c}, Dragan and Ghebleh, Mohammad and Caporossi, Gilles and Vijayakumar, Ambat and Stevanovi{\'c}, Sanja},
  journal={Computational and Applied Mathematics},
  volume={43},
  number={6},
  pages={336},
  year={2024},
  publisher={Springer}
}

@article{fournier1977sharpness,
  title={Sharpness in Young’s inequality for convolution},
  author={Fournier, John},
  journal={Pacific Journal of Mathematics},
  volume={72},
  number={2},
  pages={383--397},
  year={1977},
  publisher={Mathematical Sciences Publishers}
}

@article{sanders2012approximate,
  title={Approximate (abelian) groups},
  author={Sanders, Tom},
  journal={European Congress of Mathematics},
  year={2012}
}

@article{breuillard2014brief,
  title={A brief introduction to approximate groups},
  author={Breuillard, Emmanuel},
  journal={Thin groups and superstrong approximation},
  volume={61},
  pages={23--50},
  year={2014},
  publisher={MSRI publications Berkeley, CA}
}

@book{tointon2019introduction,
  title={Introduction to approximate groups},
  author={Tointon, Matthew CH},
  volume={94},
  year={2019},
  publisher={Cambridge University Press}
}

@article{freiman1973number,
  title={Number theoretic studies in the Markov spectrum and in the structural theory of set addition},
  author={Freiman, GA and Pigarev, VP},
  journal={Kalinin. Gos. Univ., Moscow},
  year={1973}
}

@article{shao2018almost,
  title={On an almost all version of the Balog-Szemer{\'e}di-Gowers theorem},
  author={Shao, Xuancheng},
  journal={Discrete Analysis},
  year={2019}
}

@article{gowers1998new,
  title={A new proof of Szemer{\'e}di's theorem for arithmetic progressions of length four},
  author={Gowers, William Timothy},
  journal={Geometric \& Functional Analysis GAFA},
  volume={8},
  number={3},
  pages={529--551},
  year={1998},
  publisher={Birkh{\"a}user Verlag Basel}
}

@article{balog1994statistical,
  title={A statistical theorem of set addition},
  author={Balog, Antal and Szemer{\'e}di, Endre},
  journal={Combinatorica},
  volume={14},
  number={3},
  pages={263--268},
  year={1994},
  publisher={Springer}
}

@article{tao2008product,
  title={Product set estimates for non-commutative groups},
  author={Tao, Terence},
  journal={Combinatorica},
  volume={28},
  number={5},
  pages={547--594},
  year={2008},
  publisher={Springer}
}
